\documentclass{amsart}[12pt]
\usepackage{amsmath, amsthm, amssymb, enumerate, amsfonts, mathrsfs}

\newtheorem{definition}{Definition}
\newtheorem{theorem}{Theorem}

\newtheorem{corollary}{Corollary}
\newtheorem{lemma}{Lemma}

\newcommand{\im}{\operatorname{im}}
\newcommand{\tr}{\operatorname{tr}}
\newcommand{\End}{\operatorname{End}}
\newcommand{\Der}{\operatorname{Der}}
\renewcommand{\div}{\operatorname{div}}

\newcommand{\comment}[1]{}

\title{Causal geometries, null geodesics, and gravity}
\author{Jonathan Holland}
\address{Department of Mathematics\\ 
University of Pittsburgh}
\author{George Sparling}

\dedicatory{Dedicated to Sir Roger Penrose on the occasion of his 80th birthday.}
\begin{document}
\maketitle

\begin{abstract}
The authors study a generalized notion of null geodesic defined by the Legendrian dynamics of a regular conical subbundle of the tangent bundle on a manifold.  A natural extension of the Weyl tensor is shown to exist, and to depend only on this conical subbundle.  Given a suitable defining function of the conical bundle, the Raychaudhuri--Sachs equations of general relativity continue to hold, and give rise to the same phenomenon of covergence of null geodesics in regions of positive energy that underlies the theory of gravitation.
\end{abstract}

\nocite{*}

\section{Introduction}
Soon after the development of the theory of gravitation by Albert Einstein, Hermann Weyl tried to do away with lengths and instead to base the theory entirely on angles (Afriat \cite{Afriat1}, \cite{Afriat2}, \cite{Afriat3}).  In present-day language, Weyl wanted a conformally invariant theory, so a theory invariant under the transformations $g(x) \rightarrow \omega(x)^2 g(x)$, where $g(x)$ is the Lorentzian metric of the theory and $\omega(x)$, the conformal factor,  is a positive function of the space-time point $x$.  Indeed the purely gravitational degrees of freedom of space-time are represented by the Weyl tensor (Weyl \cite{WeylTensor}), which is conformally invariant (Schouten \cite{Schouten}).  However Weyl's theory did not agree with the experiment, as was quickly pointed out by Einstein \cite{EinsteinLetter}.  Weyl eventually recast and revived his theory, by turning away from gravity and instead developing gauge theory.  In the meantime, the attempt to understand the conformal properties of gravity has led to much progress and has been a constant preoccupation of many researchers.

The present work tries to reconcile the opposing philosophies of Einstein and Weyl.  At the same time, we effectively generalize the theory of Einstein.  The reconciliation is achieved by moving the entire theory into the tangent bundle of space-time, where the focus is on the bundle of null directions and the associated null geodesic foliation of that bundle, both conformally invariant constructs.    In the Einstein theory, the bundle is defined by the vanishing of the homogeneous quadratic form in the velocities, $g(x)(v, v) = 0$, where $g(x)$ is the metric and $v$ is a tangent vector at a point $x$ of the space-time.  We generalize by allowing the null cone to be given by a general homogeneous function of the velocities, subject only to the genericity condition that its Hessian with respect to the velocities be non-degenerate (and Lorentzian in the case of space-time).  We call this structure a causal geometry.  This is somewhat similar in spirit to the construction of Finsler \cite{Finsler} geometries, although the null geodesics associated to causal geometries are intrinsically non-variational.

A key point is that the allowed conformal rescalings are vastly generalized, the function $\omega$ being replaced by a homogeneous function on the tangent bundle, so a function of $2n -1$ free variables, as opposed to the standard $n$ variables, for a space-time of $n\ge 3$ dimensions.  So the conformal transformations are on a more even footing with the metric, as compared with the standard theory, where there is one function representing the conformal transformation and $2^{-1}n(n +1)$ functions encoding the metric.  Our first main result is that there is a natural generalization of the Weyl tensor, which is proven to be invariant under the enlarged class of conformal transformations.

When extrapolating from an established physical theory, one wants to preserve as much of the structure of the old theory as possible.  One motivation for going beyond the Einstein theory is the inevitable presence of singularities in the theory, as first brilliantly proved by Sir Roger Penrose \cite{Penrose1965}, \cite{PenroseBattelle}, \cite{PenroseTDGR}.  One might wish to construct a new theory free of  singularities.  However, at least classically, the intuition behind the Penrose theorem is compelling and depends only on the attractive nature of the gravitational interaction and very little on the details of the theory. We take this intuition as being vital to the generalized theory, so we wish to generalize the Penrose singularity theorem to our case.  Examination of the proof given by Penrose shows that apart from general causal properties described in Kronheimer and Penrose \cite{KronheimerPenrose} and Geroch, Kronheimer, and Penrose \cite{GerochKronheimerPenrose}, which do not depend on the null cones being quadratic in the velocities, the only other ingredient needed for the proof to go through is apparently the Raychaudhuri--Sachs effect which predicts the existence of conjugate points for congruences of null geodesics, given that a local positive energy condition holds and that there is a point of the congruence where the divergence is negative (Raychaudhuri \cite{Raychaudhuri1955}, \cite{Raychaudhuri1957}, Sachs \cite{Sachs}).   Our second main result is that the Raychauhuri--Sachs theorem naturally generalizes to the new context, as does the Raychauhuri--Sachs effect in the Lorentzian case,  subject to a natural generalization of the local positive energy condition, so this main ingredient of the singularity theorem goes through.

While so far we are presenting the theory as a generalized theory of gravity, we see applications in many other areas of mathematics and physics, for example in the theory of elasticity.  In our case the motivation for constructing the theory came from two areas studied by us: neither of these areas is concerned with generalizing the Einstein theory.  Both involve the construction of a metric that is once degenerate and yet is not generally invariant in the degenerate direction: crudely speaking, a metric of the form $g_{ij}(x, t) dx^i dx^j$, where the (invertible) coefficient matrix $g_{ij}(x, t)$ in general depends non-trivially on the parameter $t$.  The question is what to do with this metric?

The first example comes from the theory of third-order ordinary differential equations (in general non-linear) under contact equivalence \cite{HollandSparling}.  We may write such a third-order differential equation in terms of the vanishing of an ideal of one-forms in four variables: $\{dy - pdx, dp - qdx, dq - F(x, y, p, q)dx\}$ where $p = y'$, $q = y''$ and the differential equation is $y''' = F(x, y, y', y'')$.  Here the prime denotes differentiation with respect to the variable $x$.

We may pass to the three-dimensional space of solutions of the differential equation, $\mathcal{S}$. Two points of  $\mathcal{S}$ are defined to be incident if the corresponding solutions, regarded as curves in the $(x, y, p)$-space, meet and are mutually tangent.  This incidence condition defines the null cones of an ordinary conformal structure on the  space $\mathcal{S}$, provided that a certain contact invariant of the differential equation, the W\"{u}nschmann \cite{WunschmannThesis} invariant, $\mathcal{W}$, vanishes identically. The simplest example with $\mathcal{W} = 0$, is the trivial  equation $y''' = 0$, with general solution $y = sx^2 + 2tx + u$, where the null cone is that of a flat Minkowski space with conformal structure given by $ds du - dt^2$.

\'{E}lie Cartan \cite{CartanThirdOrder} and later Shiing-Shen Chern \cite{Chern} studied the space,  $\mathcal{T}$, with co-ordinates $(x, y, p, q)$.  $\mathcal{T}$ carries a canonical direction field, $V = \partial_x + p\partial_y + q\partial_p + F(x, y, p, q)\partial_q$, such that the quotient of $\mathcal{T}$ by $V$ is the space  $\mathcal{S}$.  Chern showed that $\mathcal{T}$ carries a once-degenerate conformal metric, which is killed by $V$, such that it is also invariant under $V$ (up to scale), so passes down to $\mathcal{S}$, if and only if $\mathcal{W} = 0$.  In the case that $\mathcal{W} \ne 0$, the present authors showed that one could use the Chern metric to construct on the space $\mathcal{S}$ a null cone structure by the method of envelopes  and which reduces to the standard null cone structure in the case $\mathcal{W} = 0$.  This null cone structure is exactly that given by the incidence condition.  Thus the causal geometry is natural for this case and one wants to develop an analogue of the usual connection theory which applies in this case.  Our theory does this, although, ironically, the Weyl curvature vanishes identically, as it does in the standard case of  $\mathcal{W} = 0$, for dimensional reasons.

The simplest example with $\mathcal{W} \ne 0 $ is the differential equation $y''' = y''$.   Its solutions are $y = se^x + tx + u$, where the parameters $(s, t, u)$ are global co-ordinates for $\mathcal{S}$.  The incidence conditions are $0 = dy - pdx = e^x ds + xdt + du$  and  $0 = dp - qdx = e^x ds + dt$.  Eliminating $x$ between these equations gives the causal null cone in the form $e^{1-\frac{du}{dt}} + \frac{ds}{dt} = 0$.   This is well defined and has non-singular hessian with respect to the variables $(ds, dt, du)$, provided only that $dt\ne 0$.  It is dramatically more complicated than the case of $\mathcal{W} = 0$.

A remaining issue is to relate this work to that of the second author and Pawel Nurowski, who built a canonical conformal structure in six dimensions that encodes the geometry of the third-order equation and which, when $\mathcal{W} = 0$ reduces to a conformal structure of the type first given by Charles Feffermann \cite{Fefferman}.  When $\mathcal{W} \ne 0$, one needs a generalized Fefferman structure, applied to general parabolic geometries, as shown by Hammerl and Sagerschnig \cite{HammerlSagerschnig}.

The second example comes from the twistor theory of four-dimensional real curved space-time.  One considers the co-spin bundle of the space-time.  Points of the bundle are written $(x, \pi_{A'})$, where $\pi_{A'}$ is a two-component dual primed spinor and we use the abstract spinor and tensor indices of Penrose.  The co-spin bundle is a $\mathbb{C}^2$ bundle over space-time, so has total real dimension eight.  Briefly the co-spin bundle carries an Ehresmann connection, written $d\pi_{A'}$, with complex conjugate $d\overline{\pi}_A$, which encodes the Levi-Civita spin connection of the space-time.   Then the co-spin bundle carries a canonical degenerate real metric, $\mathcal{F}$, called the Fefferman metric \cite{Zitterbewegung}:
\[  \mathcal{F} = i \theta^a (\overline{\pi}_A d\pi_{A'} - \pi_{A'} d\overline{\pi}_A).\]
Here $\theta^a$ is the canonical one-form of the space-time.   It is not difficult to show that $\mathcal{F}$ depends only on the conformal structure of the space-time.  It is actually degenerate in two directions.  The first is that of the vector field generating the scaling transformations $\pi_{A'} \rightarrow t\pi_{A'}$, where $t $ is real and positive.  Then $\mathcal{F}$ is invariant up to scale under these transformations, so one may factor out by identifying $\pi_{A'}\ne 0$ with $t \pi_{A'}$ for any positive $t$, giving an $\mathbb{S}^3$ bundle over space-time on which $\mathcal{F}$ gives a degenerate conformal structure.  The remaining degeneracy is that of the null geodesic spray, $V = \pi^{A'} \overline{\pi}^A \partial_a$, where $\partial_a$ is the system of horizontal vector fields of the Ehresmann connection on the spin bundle, representing the Levi-Civita connection.  Then the conformal structure $\mathcal{F}$ is invariant under the null geodesic spray, if and only if the Weyl curvature vanishes, if and only if the space-time is conformally flat and in that case $\mathcal{F}$ passes down to the quotient space as a non-degenerate conformal structure, which is actually a standard Fefferman conformal structure for the space.  So the vanishing of the Weyl curvature gives the analogue of the condition of vanishing  W\"{u}nschmann condition in the case of third-order differential equations.  In the case that the Weyl curvature does not vanish, $\mathcal{F}$ restricts to give the Fefferman conformal structure of a canonical $\mathcal{CR}$ structure for the projective spin bundle over any hypersurface in the space-time, which is its defining property.  As in the case of third order ordinary differential equations, it is not clear what to do with $\mathcal{F}$ in the general case. One strategy is to use the envelope approach as in the third-order case.   Then one obtains a causal geometry on the space of null geodesics.

There is only one non-conformally flat case where the relevant calculations have been carried out in detail, that of the Kapadia \cite{KapadiaSparling} metric:
\[ g = dudv - dx^2 - u^{-1} dy^2.\]
Here the co-ordinates $(u, v, x, y) \in \mathbb{R}^4$ with $u > 0$.  The metric in null and is conformal to vacuum.    The null geodesics and the null cones are explicitly computable, with the result that the null cone of a point $(u_0, v_0, x_0, y_0)$ is given by the formula:
\[ 0 = (u - u_0)(v - v_0) - (x - x_0)^2 - \frac{2(y - y_0)^2}{u + u_0}.\]
After some calculation, one can put the Fefferman conformal structure in the form:
\[ \mathcal{F} =  2d\theta(dV + 2PdX + 4Q dY) + u^{-\frac{3}{2}} dX(dY + 3u^2 dQ) - u^{-\frac{1}{2}}dP ( 3dY + u^2 dQ).\]
In these co-ordinates, the degeneracy direction is that of the vector field $V = \frac{\partial}{\partial u}$.  Differentiating with respect to $u$, we get:
\[ - \frac{2}{3} u^{\frac{5}{2}} \mathscr{L}_V\mathcal{F} = (dX -udp)(dY - u^2 dQ). \]
The causal structure for the quotient by $V$ is obtained by eliminating the variable $u$  between the relations $\mathcal{F} = \mathscr{L}_V\mathcal{F} = 0$, so splits into two parts:
\begin{itemize}\item For the first, with $dX = udP$, we have, after a rescaling, the causal structure $\mathcal{N}_1$:
\[ \mathcal{N}_1 = d\theta(dV + 2PdX + 4QdY) \left(\frac{dX}{dP} \right)^{\frac{1}{2}} - dY dP + dX dQ \frac{dX}{dP}.\]
Then $\mathcal{N}_1$ is well defined and non-degenerate, provided $\displaystyle{\frac{dX}{dP} > 0}$.
\item For the second, with $dY = u^2dQ$, we have, after a rescaling, the causal structure $\mathcal{N}_2$:
\[ \mathcal{N}_2 = 2d\theta(dV + 2PdX + 4QdY) \left(\frac{dY}{dQ}\right)^{\frac{3}{4}}  + 4dY dX - 4dY dP \left(\frac{dY}{dQ}\right)^{\frac{1}{2}}.\]
Then $\mathcal{N}_2$ is well defined and non-degenerate, provided $\displaystyle{\frac{dY}{dQ} > 0}$.
\end{itemize}
In both cases the null geodesic equations are completely integrable by quadratures.  In both cases the tidal curvature may be calculated explicitly;  in each case, when written out, the curvature involves some 533 terms.

\section{Preliminaries}
\subsection{Secondary bundle structure on the tangent bundle}\label{secondarystructure}
Let $M$ be a smooth manifold of dimension $n\ge 2$.  The tangent bundle $TM$ of $M$ consists of pairs $(x,v)$ with $x\in M$ and $v\in TM_x$, the tangent space to $M$ at $x$.  The bundle projection $\pi_{TM}:TM\to M$ is defined by $\pi_{TM}(x,v) = x$.  The double tangent bundle $TTM$ is the tangent bundle of the tangent bundle, and consists of triples $(x,v,w)$ where $(x,v)\in TM$ and $w\in TTM_{(x,v)}$.  The bundle projection $\pi_{TTM}:TTM\to TM$ is defined by $\pi_{TTM}(x,v,w) = (x,v)\in TM$.  

In local coordinates $x^i$ of $M$, there are induced linear coordinates $v^i$ in each fiber of $TM$, defined by
$$v = v^i(v)\left.\frac{\partial}{\partial x^i}\right|_{x}.$$
Then $TTM$ also carries fiber coordinates in the $2n$-dimensional space $TTM_{(x,v)}$, denoted by $\xi^i,\nu^i$, defined at $w\in TTM_{(x,v)}$ by
$$w = \xi^i(w)\left.\frac{\partial}{\partial x^i}\right|_{(x,v)} + \nu^i(w)\left.\frac{\partial}{\partial v^i}\right|_{(x,v)}.$$

Apart from the bundle projection $\pi_{TTM}:TTM\to TM$ on the second tangent bundle, there is also another natural projection given by the differential $d\pi_{TM}:TTM\to TM$.  In the local coordinates described above,
\begin{align*}
d\pi_{TM} \frac{\partial}{\partial x^i} &= \frac{\partial}{\partial x^i}\\ 
d\pi_{TM}\frac{\partial}{\partial v^i} &= 0.
\end{align*}
The kernel of $d\pi_{TM}$ is called the {\em vertical subbundle}, and is denoted by $VTM$.  There is a natural isomorphism between $VTM$ and the pullback bundle $\pi_{TM}^{-1}TM$, given as follows.  Let $x\in M$ and $v,w\in TM_x$.  The one-parameter group $L_w(s):(x,v)\mapsto (x,v+sw)$ as $s\in\mathbb{R}$ varies, is a well-defined one-parameter group of diffeomorphisms of $TM_x$ to itself.  Denote the generator of this one-parameter group by $\overline{\lambda}_{(x,v)}(w)=L_w'(0)$.  Then $\overline{\lambda}_{(x,v)} : TM_x\to VTM_{(x,v)}$.  This is a linear isomorphism for each fixed $(x,v)\in TM$, and it depends smoothly on $(x,v)$.  So it is an isomorphism $\overline{\lambda}:\pi^{-1}_{TM}TM\to VTM$ of vector bundles over $TM$.  In coordinates,
$$\overline{\lambda}\frac{\partial}{\partial x^i} = \frac{\partial}{\partial v^i}.$$
Since $\pi^{-1}_{TM}TM = TTM/VTM$, it is convenient to compose $\overline{\lambda}$ with the quotient map $q:TTM\to TTM/VTM$ to obtain $\lambda=\overline{\lambda}\circ q:TTM\to VTM$.  Then the image and kernel of $\lambda$ are both the vertical bundle $VTM$.  As a tensor, $\lambda$ can be identified with a section of $V^0TM\otimes VTM$,  In coordinates,
$$\lambda\frac{\partial}{\partial x^i} = \frac{\partial}{\partial v^i},\qquad \lambda\frac{\partial}{\partial v^i}=0$$
and, as a tensor, $\lambda = dx^i\otimes \frac{\partial}{\partial v^i}$.  

Let $X$ be a vector in $TTM$.  Define a differential operator $D_X : C^\infty(TM)\to C^\infty(TM)$ by
$$D_X(f) =\mathscr{L}_{\lambda X}f.$$
In local coordinates, if $X=\xi^i\frac{\partial}{\partial x^i} + \nu^i\frac{\partial}{\partial v^i}$, then $D_X(f) = \xi^i\frac{\partial f}{\partial v^i}$.  Let $D : C^\infty(TM)\to \Gamma_{TM}(T^*TM)$ be the one-form valued operator
$$(Df)(X) = D_Xf.$$
If $X,Y\in\Gamma_{TM}(\pi^{-1}_{TM}TM)$ are two vector fields that are lifts of vector fields from $M$, then
$$D^2_{X,Y} = D_XD_Y=D_YD_X,$$
and $D^2_{X,Y}(f)$ depends bilinearly on $X,Y$.  Commutativity follows from the commutativity of the one parameter groups $L_X$ and $L_Y$ defined previously.  

Let $TM'$ be the tangent bundle with the zero section removed, and $\pi_{TM'} = \pi_{TM}|_{TM'}:TM'\to M$ the induced projection of $TM'$ onto $M$.  Let $SM$ be the space of oriented one-dimensional linear subspaces of $TM$.  Let $\delta_s:TM'\to TM'$ be the scaling $\delta_s(x,v) = (x,sv)$ for $s>0$, and let $H=\frac{d}{ds}\delta_s|_{s=1}$ be the homogeneity vector field.  This defines a group action of $(0,\infty)$ on $TM'$, and $SM$ is the quotient bundle of $TM'$ by the group.  Let $\pi_{SM}:SM\to M$ be the projection onto $M$.  There is a factorization $\pi_{TM'}=\pi_{SM}\circ \sigma$ where $\sigma:TM'\to SM$ is the quotient mapping.

\subsection{Fr\"{o}licher--Nijenhuis bracket}
Let $X$ be a smooth manifold and $\Omega(X)=\oplus_r \Omega^r(X)$ be the graded algebra of smooth differential forms on $X$.  A derivation of degree $k$ of $\Omega(X)$ is a real linear map $D:\Omega(X)\to\Omega(X)$ such that 
\begin{itemize}
\item $D:\Omega^r(X)\to\Omega^{r+k}(X)$
\item For any $\alpha\in\Omega^a(X)$ and $\beta\in\Omega^b(X)$, $D(\alpha\wedge\beta) = (D\alpha)\wedge\beta + (-1)^{ka}\alpha\wedge D\beta$
\end{itemize}

Let $\Der_k(\Omega(X))$ be the space of derivations of degree $k$ of $\Omega(X)$, and let $\Der(\Omega(X))=\oplus_{k\in\mathbb{Z}}\Der_k(\Omega(X))$ be the graded vector space of all derivations; this supports the structure of a graded Lie algebra, where the bracket of homogeneous elements $K\in\Der_k(\Omega(X)), L\in\Der_\ell(\Omega(X))$ is defined by
$$[K,L] = K\circ L - (-1)^{k\ell}L\circ K.$$
Extending by bilinearity to all of $\Der(\Omega(X))$, the resulting bracket is easily seen to define a graded Lie algebra:
\begin{itemize}
\item  The bracket is graded anticommutative: 
$$[K,L]= -(-1)^{k\ell}[L,K]$$
for $K\in\Der_k(\Omega(X)), L\in\Der_\ell(\Omega(X))$
\item The bracket satisfies the graded Jacobi identity:
$$(-1)^{j\ell}[J,[K,L]] + (-1)^{kj}[K,[L,J]] + (-1)^{\ell k}[L,[J,K]] =0.$$
\end{itemize}

If $v\in\Gamma_X(TX)$ is a vector field, then the insertion operator $i_v:\Omega(X)\to\Omega(X)$ is a derivation of degree $-1$.  The insertion operator extends to an operator $i_K\in \Der_{k-1}(\Omega(X))$, by defining $i_{\omega\otimes v}=\omega i_v$ for $K\in\Omega^k(X,TX)$ extending by linearity.

\begin{definition}
Let $K\in\Omega^k(X,TX)$.  Define the Lie derivative along $K$ by
$$\mathscr{L}_K = [i_K,d] = i_K\circ d + (-1)^kd\circ i_K.$$
\end{definition}

The following is proven in \cite{KMS}:

\begin{theorem}
Any derivation $D\in\Der_k(\Omega)$ can be decomposed uniquely as
$$D=\mathscr{L}_K + i_L$$
for some $K\in\Omega^k(X,TX)$ and $L\in\Omega^{k+1}(X,TX)$.
\end{theorem}

\begin{definition}
For $K\Omega^k(X,TX)$ and $L\in\Omega^\ell(X,TX)$ define the {\em Nijenhuis--Richardson bracket} $[K,L]^\wedge$ by
$$i_{[K,L]^\wedge} = [i_K,i_L].$$
Define the {\em Fr\"{o}licher--Nijenhuis bracket} $[K,L]$ by
$$\mathscr{L}_{[K,L]} = [\mathscr{L}_K,\mathscr{L}_L].$$
\end{definition}

The algebraic properties of the curvature and related quantities are most easily expressed using the Nijenhuis--Richardson bracket \cite{NijenhuisRichardson} and the Fr\"{o}licher--Nijenhuis bracket \cite{FrolicherNijenhuis}.  

\begin{lemma}
$[\mathscr{L}_K,i_L] = i_{[K,L]}-(-1)^{k\ell}\mathscr{L}_{i_LK}$
\end{lemma}

\begin{proof}

\end{proof}

\begin{lemma}\label{FrolicherDerivation}
\begin{align*}
[K,[L_1,L_2]^\wedge] &= [[K,L_1],L_2]^\wedge + (-1)^{k\ell_1}[L_1,[K,L_2]]^\wedge-\\
&\quad-\left((-1)^{k\ell_1}[i_{L_1}K,L_2]-(-1)^{(k+\ell_1)\ell_2}[i_{L_2}K,L_1]\right)
\end{align*}
\end{lemma}
\begin{proof}

\end{proof}

The following is also useful:
\begin{lemma}\label{KL}
If $K,L\in\Omega^1(X,TX)$, then
\begin{align*}[K,L](A,B) = [KA,LB]-[KB,LA] &- L([KA,B]-[KB,A]) \\
&-K([LA,B]-[LB,A])+(KL+LK)[A,B].
\end{align*}
\end{lemma}

\section{Null geodesic dynamics}
\subsection{Legendrian dynamics}
The dynamics is specified by a smooth hypersurface $\mathscr{G}$ (of dimension $2n-2$) in $SM$, with $\pi(\mathscr{G})=M$, which has the property that for any $x\in M$, the intersection $\mathscr{G}\cap SM_x$ is a smooth submanifold of the sphere $SM_x$ of dimension $n-2$.  Note that the space $\mathscr{H}=\sigma^{-1}\mathscr{G}$, a smooth hypersurface in $TM'$ invariant under the scaling $(x,v)\to (x,tv)$, also specifies the dynamics.

If $(x,v)\in\mathscr{H}$, the vertical tangent space $V\mathscr{H}_{(x,v)}$ to $\mathscr{H}$ at $(x,v)$ is the intersection of the tangent space to $\mathscr{H}$ at $(x,v)$ with the vertical space $VTM_{(x,v)}$.  So $V\mathscr{H}_{(x,v)}$ has dimension $n-1$.  The image of $V\mathscr{H}_{(x,v)}$ under the map $\overline{\lambda}^{-1}_{(x,v)}:VTM_{(x,v)}\to TM_x$ is then a subspace of $TM_x$, also of dimension $n-1$.  There is then a unique maximal subspace of $T\mathscr{H}_{(x,v)}$ that projects down under the map $d\pi_{TM'}$ to $\overline{\lambda}^{-1}_{(x,v)}(V\mathscr{H}_{(x,v)})$.  This subspace is a codimension one distribution within $T\mathscr{H}$, and therefore defines a distribution of hyperplanes on $\mathscr{H}$:
$$\Lambda_{\mathscr{H}} = T\mathscr{H}\cap d\pi_{TM'}^{-1}\left(\overline{\lambda}^{-1}(V\mathscr{H})\right).$$
This entire construction is invariant under the scalar homothety $\delta$, and so $\Lambda_{\mathscr{H}}$ descends to a distribution of $\Lambda_{\mathscr{G}}$ on $\mathscr{G}$ as well.

\begin{definition}
A {\em contact symmetry} of a distribution $\Lambda$ on a manifold $X$ is a one-parameter local group of diffeomorphisms of $X$ that preserves $\Lambda$ and whose generators are everywhere tangent to $\Lambda$.
\end{definition}

\begin{definition}
The {\em dynamics} of $\mathscr{G}$ is the space of contact symmetries of $\Lambda_{\mathscr{G}}$.
\end{definition}

\subsection{Lagrangian approach to the dynamics}
Let $\mathscr{H}$ have local defining equation $G(x,v)=0$ where $G$ is a smooth function defined over an open set $U$ of $TM'$ that is invariant under $\delta$ satisfying:
\begin{itemize}
\item $DG\not=0$ throughout $U$
\item $G$ homogeneous of some real degree $k$: $G(x,tv)=t^kG(x,v)$ for all $t>0$ and all $(x,v)\in U$.  For convenience, we shall henceforth assume that $k\not=1$.
\end{itemize}

There is a bilinear form $g_h$ on $\pi_{TM'}^{-1}TM$ defined for vector fields $X$ and $Y$ that lift vector fields on $M$ by
$$g_h(X,Y) = D^2_{X,Y} G.$$
The definition is independent of the choice of lift of $X$ and $Y$, and it is bihomogeneous under rescalings $X\to (\pi_{TM'}^*a)X$ and $Y\to (\pi_{TM'}^*b)Y$ where $a,b$ are functions on $M$, and so it gives rise to a bilinear form.  The $h$ here stands for ``horizontal'', a reflection of the fact that $g_h$ is a section of $V^0TM\otimes V^0TM$.  Applying $\overline{\lambda}^{-1}$ yields a bilinear form $g_v$ in $V^*TM\otimes V^*TM$:
$$g_v(X,Y) = g_h(\overline{\lambda}^{-1}(X),\overline{\lambda}^{-1}(Y)).$$
Here the subscript $v$ means ``vertical'', since $g_v$ is a bilinear form on $VTM$.  In coordinates,
\begin{align*}
g_v &= \frac{\partial^2 G}{\partial v^i\partial v^j}dv^i\otimes dv^j = g_{ij}dv^i\otimes dv^j  \\
g_h &= \frac{\partial^2 G}{\partial v^i\partial v^j}dx^i\otimes dx^j = g_{ij}dx^i\otimes dx^j.
\end{align*}

\begin{lemma}
Let $\alpha\in \Gamma_{TM'}(T^*TM')$ be the differential form $\alpha = DG$.  Then, on restricting to $\mathscr{H}$, the distribution $\Lambda_{\mathscr{H}}$ is the annihilator of $\alpha$ in $T\mathscr{H}$.
\end{lemma}
\begin{proof}
By the assumption that $DG\not=0$, the image of $T\mathscr{H}$ under $\alpha$ is always one-dimensional, and so the annihilator of $\alpha$ is a distribution of hyperplanes in $T\mathscr{H}$.  Suppose $\alpha(X)=0$ for $X\in T\mathscr{H}$.  Then, by definition of the $D$ operator, $\lambda X\lrcorner\,dG=0$.  So $\lambda X\in V\mathscr{H}$.  That is, $X\in \lambda^{-1}(V\mathscr{H})$ as required.
\end{proof}

\begin{lemma}\label{gdalpha}
$g_h(X,Y)=2d\alpha(\lambda X,Y)$
\end{lemma}
\begin{proof}
Both sides vanish if either $X$ or $Y$ is vertical, so it is sufficient to establish the lemma under the additional assumption that $X$ and $Y$ are lifts of vector fields from $M$.  Since $\alpha(\lambda X)=0$,
\begin{align*}
2d\alpha(\lambda X,Y) &= (\mathscr{L}_{\lambda X}\alpha)(Y)=\lambda X (\alpha(Y)) - \alpha([\lambda X,Y]) \\
&= D^2_{X,Y}G - \alpha([\lambda X,Y]) = g_h(X,Y)- \alpha([\lambda X,Y])
\end{align*}
But if $X$ and $Y$ are lifts of vector fields, then $[\lambda X,Y]$ is vertical, and so $\alpha([\lambda X,Y])=0$.
\end{proof}

Assume henceforth that the bilinear form $g_v$ is nondegenerate.  This assumption is justified in part by

\begin{lemma}
The bilinear form $g_v$ is nondegenerate if and only if $d\alpha$ is a symplectic form on a neighborhood of $\mathscr{H}$ in $TM'$.
\end{lemma}

\begin{proof}
The subspace $VTM'$ is an isotropic space for $d\alpha$.  Choose a complementary space $HTM'$ in $TTM'$.  Then $d\alpha$ induces a bilinear form on $VTM'\times HTM'$ and $d\alpha(X,Y) = 2g_v(X,\lambda Y)$.
\end{proof}

In coordinates,
$$\alpha = \frac{\partial G}{\partial v^i}dx^i.$$
Set $p_i=\partial G/\partial v^i$.  By the nondegeneracy of $g$, the Jacobian matrix $\partial p_i/\partial v^j$ is nonsingular, and so this defines a new set of (local) coordinates on $TM'$.  In the new coordinates,
$$\alpha = p_idx^i.$$
These are the ``canonical coordinates'' for the dynamical system.

The symplectic form $d\alpha$ allows us to define the Poisson bracket of two functions $f_1,f_2$ (in a neighborhood of $\mathscr{H}$) by
$$\{f_1,f_2\} = (d\alpha)^{-1}(df_1,df_2).$$
This satisfies the usual rules:
\begin{itemize}
\item $\{f_1,f_2\}=-\{f_2,f_1\}$
\item $\{f_1,\{f_2,f_3\}\} + \{f_2,\{f_3,f_1\}\}+\{f_3,\{f_1,f_2\}\}=0$
\item $\{f_1,c\}=0$ if $c$ is constant
\item $\{f_1,f_2+f_3\} = \{f_1,f_2\} + \{f_1,f_3\}$
\item $\{f_1,f_2f_3\}=\{f_1,f_2\}f_3 + \{f_1,f_3\}f_2$
\end{itemize}
The last three properties imply that the operator $\{f_1,-\}:f_2\mapsto \{f_1,f_2\}$ is a derivation on smooth functions, and therefore corresponds to a vector field on $M$.

In the canonical coordinates,
$$ \{f,-\} = \frac{\partial f}{\partial p_i}\frac{\partial}{\partial x^i} - \frac{\partial f}{\partial x^i}\frac{\partial}{\partial p_i}.$$

As we are interested in the intrinsic geometry of $\mathscr{H}$, we shall consider the pullback of $\alpha$ to $\mathscr{H}$.  

\begin{lemma}\label{DarbouxH}
When pulled back to $\mathscr{H}$, $\alpha$ has Darboux rank $2n-3$:
$$\alpha\wedge (d\alpha)^{n-2}\not=0, \quad (d\alpha)^{n-1}=0.$$
\end{lemma}

\begin{proof}
The fibers of $\mathscr{H}\to M$ are $n-1$ dimensional, and the bilinear form $g_v$ on $V\mathscr{H}$ is annihilated by the generators of scaling up the fiber.  So on $V\mathscr{H}$, $g_v$ has rank $n-2$.  By the argument in the previous lemma, $(d\alpha)^{n-1}=0$.   However, applying the previous argument to $\Lambda_{\mathscr{H}}=\alpha^0$, and choosing a complement for this in $T\mathscr{G}$ gives $\alpha\wedge (d\alpha)^{n-2}\not=0$.
\end{proof}

The first main result uses the Darboux theorem:
\begin{lemma}
Let $M$ be a manifold of dimension $2n-1$ and $\alpha$ a one-form of Darboux rank $2r-1$.  Then the space of vector fields $X$ such that
\begin{equation}\label{contact}
X\lrcorner\,\alpha=0,\quad \alpha\wedge\mathscr{L}_X\alpha = 0
\end{equation}
forms an integrable distribution of rank $2(n-r)$.
\end{lemma}
\begin{proof}
Using the usual version of Darboux' theorem, there exists a coordinate system $x,y^1,\dots, y^{n-1}, p_1,\dots,p_{n-1}$ on $M$ such that
$$\alpha = dx + \sum_{i=1}^{r-1} p_i\,dy^i,$$
with the last $(n-r)$ $p$'s and $y$'s not participitating.  Hence the vector fields $\partial/\partial p_i$ and $\partial/\partial y^i$ for $i=r,\dots,n-1$ form an integrable distribution of rank $2(n-r)$ satisfying \eqref{contact}.

Now, note that any $X$ satisfying \eqref{contact} must also satisfy
$$X\lrcorner (\alpha\wedge (d\alpha)^{r-1})=0.$$
But $\alpha\wedge (d\alpha)^{r-1}=\pm dx\wedge dp_1\wedge\cdots\wedge dp_{r-1}\wedge dy^1\wedge\cdots\wedge dy^{r-1}$ is annihilated by $X$ if and only if $X$ is a linear combination of $\partial/\partial p_i$ and $\partial/\partial y^i$ for $i=r,\dots,n-1$.
\end{proof}

\begin{theorem}\label{dynamical}
The dynamical vector fields on $\mathscr{H}$ are spanned as a $C^\infty$ module by $H$ (the generator of the scaling symmetry of $\mathscr{H}$) and the vector field $V=(k-1)\{G,-\}$ restricted to $\mathscr{H}$.
\end{theorem}

The particular normalization of $V$ ensures that it defines a spray; see Lemma \ref{spray} below.

\begin{proof}
By Lemma \ref{DarbouxH}, there are exactly two linearly independent dynamical vector fields at every point.  Note that $V$ is tangent to $\mathscr{H}$ since $V(G)=(k-1)\{G,G\}=0,$ and $H$ is tangent to $\mathscr{H}$ since $\mathscr{H}$ is invariant under the scaling action.  These are linearly independent, since $V(\pi_{\mathscr{H}}^*f)=(k-1)\{G,\pi_{\mathscr{H}}^*f\}$ is nonzero for some smooth function $f$ on $M$, but $H(\pi_{\mathscr{H}}^*f)=0$ for all such $f$.

Now, note that $H$ satisfies $H\lrcorner\alpha=0$ (since $\lambda H=0$).  If $X$ is a lift of a vector field on $M$, then $\delta_s\lambda X=s^{-1}\lambda X$.  Differentiating gives $\mathscr{L}_H(\lambda X) = -\lambda X$.  For such a vector field $X$, $[H,X]$ is vertical and so $\lambda [H,X]=0$.  It follows that $\alpha([H,X])=0$, and therefore
\begin{align*}
(\mathscr{L}_H \alpha)(X) &=\mathscr{L}_H(\lambda X\lrcorner dG) - \alpha([H,X])\\
&=-\lambda X\lrcorner\, dG + \lambda X\lrcorner\, \mathscr{L}_H dG\\
&=(k-1)\lambda X\lrcorner\, dG = (k-1)\alpha(X).
\end{align*}

Finally, in a neighborhood of $\mathscr{H}$, $V$ is characterized by
$$V\lrcorner d\alpha = -(k-1)dG.$$
Pulling back to $G=0$ gives $V\lrcorner d\alpha=0$.  By the previous calculation, $H\lrcorner\, d\alpha=(k-1)\alpha$.  Hence
\begin{align*}
V\lrcorner\,\alpha &= (k-1)^{-1}V\lrcorner\, H\lrcorner\, d\alpha = - (k-1)^{-1}H\lrcorner\, V\lrcorner\, d\alpha \\
&= H\lrcorner\, dG = k G
\end{align*}
which also vanishes on $\mathscr{H}$.
\end{proof}

In coordinates, the dynamical vector fields are
\begin{align*}
H &= v^i\frac{\partial}{\partial v^i}\\
V &= v^i\frac{\partial}{\partial x^i} + u^i\frac{\partial}{\partial v^i}, \quad u^ig_{ij} = \frac{\partial G}{\partial x^j} - v^i\frac{\partial^2G}{\partial x^i\partial v^j}.
\end{align*}
A direct derivation of the previous results in coordinates is given in an appendix.

The integral curve of the vector field $V$ through a point $(x,v)$ projects to a curve in $M$ whose initial velocity is $v$.  That is, $V$ is a semispray.

\begin{lemma}\label{spray}
$V$ is a spray:
\begin{itemize}
\item $[H,V]=V$
\item $H=\lambda V$
\end{itemize}
\end{lemma}
\begin{proof}
The calculations in the proof of the preceding lemma give
\begin{align*}
[H,V]\lrcorner\,d\alpha &= \mathscr{L}_H(V\lrcorner\,d\alpha) - V\lrcorner(\mathscr{L}_Hd\alpha)\\
&=k(V\lrcorner\, d\alpha) - (k-1)V\lrcorner\,d\alpha = V\lrcorner\,d\alpha,
\end{align*}
so $[H,V]=V$.

For the second property, the definition of $g_h$ implies
$$d\alpha(X,\lambda Y) = \frac{1}{2}g_h(X,Y) = d\alpha(Y,\lambda X) = -d\alpha(\lambda X,Y).$$
In particular, with $X=V$,
$$d\alpha(\lambda V, Y) = -d\alpha(V,\lambda Y) = (k-1)dG(\lambda Y) = (k-1)\alpha(Y) = d\alpha(H,Y).$$
This is true for all $Y$ and so $H=\lambda V$ by nondegeneracy of $d\alpha$.
\end{proof}

Passing down to the sphere bundle $SM$, only the dynamical vector field $V$ survives, up to an overall positive scale, since $H$ is in the kernel of $d\sigma:TTM\to TSM$.  This gives a foliation of $\mathscr{G}$ by the dynamical curves, the (maximally extended) trajectories of $V$.  These dynamical curves are called {\em null geodesics}.  The space of null geodesics, denoted by $\mathscr{N}$, has dimension $2n-3$.  The distribution $\Lambda_{\mathscr{G}}$ is Lie derived along the dynamical vector fields, and so descends to a codimension one distribution on $\mathscr{N}$.  This distribution is a contact structure since the relation $\alpha_{\mathscr{G}}\wedge (d\alpha_{\mathscr{G}})^{n-2}\not=0$, valid for any nonzero $\alpha_{\mathscr{G}}$ in the annihilator of $\Lambda_{\mathscr{G}}$, also descends to the quotient.

The null geodesics are naturally oriented, since at each point $p$ of $\mathscr{G}$, the vector field $V$ descends to a ray through the origin in $T_p\mathscr{G}$, which is oriented.  The bundle $\mathscr{G}$ is {\em time oriented} if and only if the space of oriented null geodesics is the disjoint union of two components, $\mathscr{N}=\mathscr{N}^+\cup\mathscr{N}^-$, such that the oriented null geodesic through $(x,v)$ lies in $\mathscr{N}^\pm$ if and only if the oriented null geodesic through $(x,-v)$ lies in $\mathscr{N}^\mp$.  Then the elements of $\mathscr{N}^+$ are called future oriented and the elements of $\mathscr{N}^-$ are called past oriented.

On $\mathscr{H}$, the integral curves of $V$ are called {\em affinely parametrized} null geodesics.  These carry a natural parametrization up to a translation, since they are the integral curves of a single vector field.  This natural parametrization requires having a particular defining function $G$ for $\mathscr{H}$, although the definition of (unparametrized, oriented) null geodesics on $\mathscr{G}$ does not.

\section{Ehresmann connection}
\subsection{The Ehresmann connection on $\mathscr{H}$}
The purpose of this section is to establish the following:
\begin{theorem}\label{ehresmann}
There exists a unique operator $P:TTM'\to VTM'$ satisfying  for all $X,Y\in TTM$:
\begin{enumerate}
\item $(\mathscr{L}_V g_h)(X,Y) = g_v\left(P X,\lambda Y\right) + g_v\left(P Y,\lambda X\right)$
\item $2d\alpha(X,Y) = g_v\left(P X,\lambda Y\right) - g_v(PY,\lambda X)$
\end{enumerate}
This operator defines an Ehresmann connection on $TM'$, meaning that it has maximal rank and satisfies $P=P\circ P$.  Furthermore,
\begin{equation}
P = \frac{1}{2}\left(\operatorname{Id}_{TTM'}+\mathscr{L}_V\lambda\right).
\end{equation}
The restriction of $P$ to $\mathscr{H}$ is also an Ehresmann connection on $\mathscr{H}$: $P(T\mathscr{H})=V\mathscr{H}$.
\end{theorem}

(Recall that $\lambda$ is a section of $V^0TM'\otimes VTM'$, $g_h$ is a section of $V^0TM'\otimes V^0TM'$, and $g_v$ is a section of $V^*TM'\otimes V^*TM'$.)

The proof is broken down into several lemmas.

\begin{lemma}\label{lambdalambda}
The Fr\"olicher--Nijenhuis bracket of $\lambda$ with itself is zero: $[\lambda,\lambda]=0$.  Thus if $X,Y\in\Gamma_{TM'}(TTM')$, then
$$\lambda([\lambda X, Y] + [X,\lambda Y])= [\lambda X,\lambda Y].$$
Moreover $[V,\lambda Y] = -Y\pmod{VTM'}$ for all vector fields $Y$.
\end{lemma}
\begin{proof}
If $X,Y\in\Gamma_{TM'}(TTM')$, then
$$\frac{1}{2}[\lambda,\lambda](X,Y)=[\lambda X,\lambda Y] - \lambda([\lambda X, Y] + [X,\lambda Y])$$ 
The right-hand side vanishes if $X$ or $Y$ is a section of $VTM'$, since $VTM'$ is an integrable distribution on which $\lambda$ vanishes.  Thus it suffices to prove that it vanishes if $X$ and $Y$ are both lifts of vector fields from $M$.  In that case, if $f$ is a function on $M$, then
$$[\lambda X,Y] \pi^*_{TM'}f = (\lambda X)Y \pi^*_{TM'}f = (\lambda X)\pi^*_{TM'}((d\pi_{TM'}Y)f) = 0.$$
Hence $[\lambda X,Y]\in \Gamma_{TM'}(VTM')$; likewise $[X,\lambda Y]\in \Gamma_{TM'}(VTM')$.  So $\lambda([\lambda X, Y] + [X,\lambda Y])=0$.  Finally, $[\lambda X,\lambda Y]=0$ as well for $X,Y$ lifts of vector fields on $M$, since the one-parameter groups $L_X$ and $L_Y$ (defined in \S \ref{secondarystructure}) commute in that case.

It remains only to show that $[V,\lambda Y] = -Y\pmod{VTM'}$ for all vector fields $Y$.  Taking $X=V$ in the first part gives
$$\lambda([\lambda V,Y]+[V,\lambda Y]) = [\lambda V,\lambda Y].$$
But $\lambda V=H$, so rearranging gives
$$\lambda [V,\lambda Y] = (\mathscr{L}_H\lambda)(Y) = -Y$$
as claimed.
\end{proof}

\begin{lemma}\label{Ehresmann1}
Any operator $P:TTM' \to VTM'$ satisfying property (1) has maximal rank and satisfies $P\circ P=P$.
\end{lemma}

\begin{proof}
In view of the fact that $\operatorname{im} P\subset VTM'$ by assumption, it is enough to show that $P(\lambda(X))=\lambda(X)$ for all $X$.  This then proves that the range of $P$ is equal to the vertical tangent space, and that $P$ acts as the identity on its range.  Therefore $P\circ P=P$, and $P$ has maximal rank.

To prove the claimed identity, (1) gives
\begin{align*}
(\mathscr{L}_Vg_h)(\lambda X,Y) &= g_v(P\lambda X,\lambda Y)+ g_v(PY,\lambda\lambda X) \\
&= g_v(P\lambda X,\lambda Y)
\end{align*}
since $\lambda \lambda X=0$.  Expanding the left-hand side,
\begin{align*}
(\mathscr{L}_Vg_h)(\lambda X,Y)&= V(g_h(\lambda X,Y)) - g_h([V,\lambda X],Y) - g_h(\lambda X,[V,Y])\\
&=- g_h([V,\lambda X],Y)
\end{align*}
since every vertical direction lies in the kernel of $g_h$.  But, by Lemma \ref{spray}, $[V,\lambda X]=-X\pmod{VTM}$, and therefore
$$-g_h([V,\lambda X],Y) = g_h(X,Y).$$

Putting these together,
\begin{equation}\label{Plambda}
g_h(X,Y) = g_v(P\lambda X,\lambda Y).
\end{equation}
This is true for all $X,Y$, and so $P\lambda X=\lambda X$, as claimed.
\end{proof}

\begin{lemma}\label{Ehresmann2}
There exists a unique $P\in \Gamma_{TM'}(T^*TM'\otimes VTM')$ satisfying conditions (1) and (2).
\end{lemma}

\begin{proof}
If (1) and (2) hold, then
$$g_v(PX,\lambda Y) = \frac{1}{2}\left((\mathscr{L}_Vg_h)(X,Y) + d\alpha(X,Y)\right).$$
By the non-degeneracy of $g_v$ and the fact that $\lambda:TTM' \to VTM'$ has maximal rank, this admits at most a unique solution $P(X)$ valid for all $Y$.  To prove existence, it is enough to show that the kernel of
$$Y\mapsto (\mathscr{L}_Vg_h)(X,Y) + d\alpha(X,Y)$$
contains the kernel of $\lambda$, which is also the image of $\lambda$, $VTM'$.  So consider
\begin{align*}
(\mathscr{L}_Vg_h)(X,\lambda Y) + d\alpha(X,\lambda Y) &= -g_h(X,[V,\lambda Y]) + d\alpha(X,\lambda Y)\\
&=g_h(X,Y) - g_h(X,Y) = 0
\end{align*}
where we have used the fact that $V$ is a spray in simplifying the first term, and the definition of $g_h$ in simplifying the second term.
\end{proof}

\begin{lemma}
The unique connection $P$ satisfying (1) and (2) is given explicitly by
$$P=\frac{1}{2}\left(\operatorname{Id}_{TTM'}+\mathscr{L}_V\lambda\right).$$
\end{lemma}
\begin{proof}
Let $P$ be the connection characterized by (1) and (2) and let $P_1=\frac{1}{2}\left(\operatorname{Id}_{TTM'}+\mathscr{L}_V\lambda\right)$.  We will show that $P_1$ acts as the identity on $VTM'$, and that $\ker P_1=\ker P$.  The first claim follows at once from
$$(\mathscr{L}_V\lambda)(\lambda X) = [V,\lambda^2X]-\lambda[V,\lambda X] = \lambda X.$$

For the second claim, $X\in\ker P_1$ if and only if $(\mathscr{L}_V\lambda)(X) = -X,$
or, equivalently,
$$-X = [V,\lambda X] - \lambda [V,X].$$
Now $X\in\ker P$ if and only if
$$(\mathscr{L}_Vg_h)(X,Y) = -2d\alpha(X,Y)$$
for all $Y\in TTM'$.  Note
\begin{align*}
(\mathscr{L}_Vg_h)(X,Y) &= V(g_h(X,Y)) - g_h([V,X],Y) - g_h(X,[V,Y])\\
&=V(d\alpha(\lambda X,Y)) - 2d\alpha(\lambda[V,X],Y) - 2d\alpha(\lambda X,[V,Y])\\
&= 2d\alpha([V,\lambda X],Y) -2d\alpha(\lambda[V,X],Y)\\
&= 2d\alpha((\mathscr{L}_V\lambda)(X),Y).
\end{align*}
If $X\in\ker P_1$, then this last display reduces to $-2d\alpha(X,Y)$, and so $X\in\ker P$.  Conversely, if $X\in\ker P$, then the same calculation shows that $2d\alpha((\mathscr{L}_V\lambda)(X),Y) = -2d\alpha(X,Y)$ for all $Y$, and hence $(\mathscr{L}_V\lambda)(X)=-X$ by the nondegeneracy of $d\alpha$, and so $X\in\ker P_1$.
\end{proof}

In coordinates,
$$P = \left(dv^i + U_j^idx^j\right)\otimes \frac{\partial}{\partial v^i}\qquad U_j^i = -\frac{1}{2}\frac{\partial u^i}{\partial v^j},\qquad u^ig_{ij}=\frac{\partial G}{\partial x^j} - v^i\frac{\partial^2G}{\partial x^i\partial v^j}.$$
The horizontal lift of the coordinate vector fields $\partial/\partial x^i$ are
$$h(\partial/\partial x^i) = (I-P)(\partial/\partial x^i) = \frac{\partial}{\partial x^i} - U_i^j\frac{\partial}{\partial v^j}.$$

\begin{lemma}
The operator $P:TTM'\to VTM'$ satisfying (1) and (2) is such that on $\mathscr{H}$, $P(T\mathscr{H}) = V\mathscr{H}$.
\end{lemma}
\begin{proof}
Let $X$ be a vector field in $T\mathscr{H}$ that Lie commutes with $V$.  Since $\lambda V=H$,
$$g_v(P(X),H) = g_v(P(X),\lambda V) = \frac{1}{2}\left(d\alpha(X,V) + (\mathscr{L}_Vg_h)(X,V)\right).$$
It is sufficient prove that both terms of the right-hand side are zero.  By the calculations in the proof of Theorem \ref{dynamical}
$$d\alpha(X,V) = (k-1)dG(X) = 0$$
since $X$ is tangent to $\mathscr{H}$.  Also, since $X,V$ commute by hypothesis,
$$ (\mathscr{L}_Vg_h)(X,V) = V\left(g_h(X,V)\right).$$
Now $g_h(X,V) = -2d\alpha(X,\lambda V) = -2d\alpha(X,H) = (k-1)\alpha(X)$, and thus
$$V\left(g_h(X,V)\right)=(k-1)V(\alpha(X)) = (k-1)(\mathscr{L}_V\alpha)(X) = (k-1)dG(X) = 0.$$
\end{proof}

\section{Curvature}
\subsection{Tidal force}
\begin{theorem}\label{tidalforce}
Let $P$ be the connection of Lemmas \ref{Ehresmann1} and \ref{Ehresmann2}.  Then there exists $S\in\Gamma_{TM'}(V^0TM'\otimes V^0TM')$ such that $\frac{1}{2}(\mathscr{L}_V^2g_h)(X,Y) = g_v(P(X),P(Y)) + S(X,Y)$ for all $X,Y\in TTM'$.  Conversely, $P$ is the unique operator such that
\begin{enumerate}
\item $(\mathscr{L}_Vg_h)(X,Y) = g_v(P(X),Y) + g_v(X,P(Y))$ for all $X,Y\in TTM'$.
\item There exists $S\in V^0TM'\otimes V^0TM'$ such that $\frac{1}{2}(\mathscr{L}_V^2g_h)(X,Y) = g_v(P(X),P(Y)) + S(X,Y)$ for all $X,Y\in TTM'$.
\end{enumerate}
\end{theorem}

The symmetric tensor $S$ is called the {\em tidal force tensor}.

\begin{proof}
For the first claim, it is enough to show:
\begin{enumerate}
\item $(\mathscr{L}_V^2g_h)(\lambda X,\lambda Y) = 2g_v(\lambda X,\lambda Y)$ 
\item $(\mathscr{L}_V^2g_h)(\lambda X,Y)=0$ for all $X$ and $Y$ in the kernel of $P$.
\end{enumerate}
Indeed, assuming these are both true, decomposing two vectors $X=X_h+X_v$ and $Y=Y_h+Y_v$ into $\ker P$ and $\im P$ components,
$$(\mathscr{L}_V^2g_h)(X_h+X_v,Y_h+Y_v)  - g_v(X_v,Y_v) = (\mathscr{L}_V^2g_h)(X_h,Y_h)$$
which defines $S(X,Y)$.

For (1), since $(\mathscr{L}_Vg_h)(\lambda X,\lambda Y) = 0$ and $g_h(\lambda X,Z)=g_h(Z,\lambda Y)=0$ for all $Z$,
\begin{align*}
(\mathscr{L}_V^2g_h)(\lambda X,\lambda Y) &= - (\mathscr{L}_Vg_h)([V,\lambda X],\lambda Y) - (\mathscr{L}_Vg_h)(\lambda X,[V,\lambda Y])\\
&=2g_h([V,\lambda X],[V,\lambda Y]) = 2g_h(X,Y) = 2g_v(\lambda X,\lambda Y).
\end{align*}

For (2), $Y$ is in the kernel of $P$ if and only if
$$(\mathscr{L}_Vg_h)(Y,Z) + 2d\alpha(Y,Z) = 0$$
for all $Z$.  Hence
\begin{align*}
(\mathscr{L}_V^2g_h)(\lambda X,Y) &= V((\mathscr{L}_Vg_h)(\lambda X,Y)) - (\mathscr{L}_Vg_h)([V,\lambda X],Y) - (\mathscr{L}_Vg_h)(\lambda X, [V,Y])\\
&= V(g_v(P(\lambda X),\lambda Y)) +d\alpha(Y,[V,\lambda X]) - g_v(P(\lambda X), \lambda[V,Y])\\
&= V(g_h(X,Y)) +d\alpha(Y,[V,\lambda X]) - g_h(X, [V,Y])\\
&= (\mathscr{L}_Vg_h)(X,Y) + g_h([V,X],Y) +d\alpha(Y,[V,\lambda X])\\
&= g_h([V,X],Y) +2d\alpha(Y,[V,\lambda X])
\end{align*}
since $d\alpha(X,Y)=0$ for $X,Y\in\ker P$.  Now,
\begin{align*}
[V,\lambda X] &= (\mathscr{L}_V\lambda)(X) + \lambda [V,X] = (2P-\operatorname{Id})X + \lambda [V,X] \\
&=-X + \lambda[V,X].
\end{align*}
So, continuing the above calculation gives
$$(\mathscr{L}_V^2g_h)(\lambda X,Y) = g_h([V,X],Y) +d\alpha(Y,\lambda[V,X]) = 0$$
where we have used again the fact that $d\alpha(X,Y)=0$ along with Lemma \ref{gdalpha}.

For the converse statement, let $\mathscr{P}$ be the affine space consisting of all operators $P:TTM'\to VTM'$ satisfying
$$(\mathscr{L}_Vg_h)(X,Y) = g_v(P(X),\lambda(Y)) + g_V(\lambda(X),P(Y)).$$
By Lemma \ref{Ehresmann1}, any such $P$ satisfies $P\circ P=P$, and so defines a projection onto $VTM'$.  Any such operator is completely determined by its kernel.  But for the $P$ satisfying Lemma \ref{Ehresmann2}, it follows from the first part of the lemma that
$$\ker P = \bigcap_{X\in VTM'} \ker [(\mathscr{L}_V^2g_h)(X, -)].$$
\end{proof}
\subsection{Curvature of the connection}
The curvature of the Ehresmann connection $P$ is the section of $V^0TM'\otimes V^0TM'\otimes VTM'$ defined by
$$R(X,Y) = P[(\operatorname{Id}-P)(X),(\operatorname{Id}-P)(Y)].$$

Since the vertical bundle is integrable, this evaluates to
$$R(X,Y) = [PX,PY] - P([PX,Y] + [X,PY]) + P[X,Y].$$
Equivalently, this can be re-expressed in terms of the Fr\"{o}hlicher--Nijenhuis bracket \cite{KMS}, by
$$R = \frac{1}{2}[P,P].$$

\begin{lemma}
The Bianchi identity holds $[P,R]=0$.  That is,
$$[PX,R(Y,Z)] + P[R(X,Y),Z] + R([X,Y],Z) - R([PX,Y],Z) + R([PY,X],Z) + \operatorname{cyclic} = 0.$$
\end{lemma}

Define $S^\lambda\in T^*TM'\otimes VTM'$ to be the unique tensor such that
$$g_v(S^\lambda X,\lambda Y) = S(X,Y)$$
for all $X,Y$, where $S$ is the tidal force tensor of Theorem \ref{tidalforce}.  The curvature determines the tidal force, and vice versa:
\begin{theorem}\label{tidalforceandcurvature}
The curvature and tidal force are related by
$$S^\lambda X=-R(V,X)$$
for all $X,Y\in TTM'$.  Moreover,
$$[\lambda, S^\lambda] = -\frac{3}{2}R.$$
\end{theorem}

The following lemma is of interest in its own right:
\begin{lemma}\label{tidalforcelemma}
The tidal force tensor satisifies
$$S(X,Y) = \frac{1}{2}g_v(P(\mathscr{L}_V^2\lambda)X,\lambda Y).$$
So $S^\lambda = \frac{1}{2}P(\mathscr{L}^2_V\lambda).$  Moreover, if $X,Y\in\ker P$, then
\begin{equation}\label{SkerP}
S(X,Y)=-g_v(P[V,X],\lambda Y).
\end{equation}
\end{lemma}

\begin{proof}
Note first the operator identities
\begin{equation}\label{Pidentities}
P(\mathscr{L}_V\lambda)=P,\quad (\mathscr{L}_V\lambda)\lambda = \lambda.
\end{equation}
Thus
$$P(\mathscr{L}_V^2\lambda)\lambda X = P[V,(\mathscr{L}_V\lambda)\lambda X] - P (\mathscr{L}_V\lambda)[V,\lambda X] = 0.$$

So it is sufficient to establish the first statement of the lemma under the additional hypothesis that $X,Y \in\ker P$.  In that case
$$X= -(\mathscr{L}_V\lambda)X$$
so
\begin{align*}
P(\mathscr{L}_V^2\lambda)X &= P[V,(\mathscr{L}_V\lambda)X]-P(\mathscr{L}_V\lambda)[V,X]\\
&= -P[V,X] - P[V,X] = -2P[V,X].
\end{align*}
So to prove the first part of the lemma, it is enough to show \eqref{SkerP}.

Since $X,Y\in\ker P$, it follows by Theorem \ref{tidalforce} that $2S(X,Y)=(\mathscr{L}_V^2g_h)(X,Y)$.  Now,
\begin{align*}
2S(X,Y) = (\mathscr{L}_V^2g_h)(X,Y)&= V\left(\mathscr{L}_Vg_h(X,Y)\right)-\mathscr{L}_Vg_h([V,X],Y)-\mathscr{L}_Vg_h(X,[V,Y])\\
&= -2d\alpha([V,X],Y)+2d\alpha(X,[V,Y])\\
&=2V(d\alpha(X,Y)) - 4d\alpha([V,X],Y)\\
&= -4d\alpha([V,X],Y)
\end{align*}
since $d\alpha(X,Y)=0$ for $X,Y\in\ker P$ by Theorem \ref{ehresmann}.  Thus, applying Theorem \ref{ehresmann} twice more,
$$2S(X,Y)=-4d\alpha([V,X],Y) = -2(\mathscr{L}_Vg_h)([V,X],Y) =-2g_v(P[V,X],\lambda Y).$$
\end{proof}

\begin{proof}[Proof of Theorem \ref{tidalforceandcurvature}]
For the first identity, it is enough to show that $(\mathscr{L}_V^2g_h)(X,Y)=-2g_v(R(V,X),\lambda Y)$ for all $X,Y\in\ker P$.  Since $V\in\ker P$ as well, $R(V,X)=P[V,X]$, and so the last part of Lemma \ref{tidalforcelemma} gives
$$S(X,Y) = -g_v(R(V,X),\lambda Y),$$
as claimed.

Now, by the first part of the theorem, $S^\lambda = -\tfrac12 i_VR$.  So, by Lemma \ref{FrolicherDerivation},
\begin{align*}
2[\lambda,S^\lambda] &= -[\lambda,[V,R]^\wedge] \\
&= -\left([[\lambda, V],R]^\wedge + [V,[\lambda,R]]^\wedge -[i_V \lambda,R]+[i_R\lambda,V]\right).
\end{align*}
Now $i_R\lambda=0$, $i_V\lambda=\lambda V = H$, and
\begin{align*}
[H,R] &= [[H,P],P]+[P,[H,P]] = 0\\
[[\lambda, V],R]^\wedge &= -[\mathscr{L}_V\lambda,R]^\wedge = -[2P-\operatorname{Id},R]^\wedge\\
&= -2[P,R]^\wedge + [\operatorname{Id},R]^\wedge = 2R + R = 3R.
\end{align*}
So
\begin{equation}\label{lambdaSlambda}
[\lambda,S^\lambda] = -(3R +[V,[\lambda,R]]^\wedge).
\end{equation}
Now, we claim that $[\lambda,R]=0$.  By the graded Jacobi identity,
$$[\lambda,R] = \frac{1}{2}[\lambda,[P,P]] = -[P,[\lambda,P]].$$
From $P=\tfrac12(\operatorname{Id}+\mathscr{L}_V\lambda)$,
$$[\lambda,P] = \frac{1}{2}\left([\lambda,\operatorname{Id}] + [\lambda, [V,\lambda]]\right)= \frac{1}{2}[\lambda, [V,\lambda]]$$
since $[\lambda,\operatorname{Id}]=0$ by Lemma \ref{KL}.  The graded Jacobi identity applied once more gives
$$-[\lambda, [V,\lambda]]+[V,[\lambda,\lambda]] + [\lambda,[\lambda,V]] =0.$$
But $[\lambda,\lambda]=0$ by Lemma \ref{lambdalambda}, and $[V,\lambda]=-[\lambda,V]$.  So $[\lambda,[V,\lambda]] = 0$, and therefore $[\lambda,R]=0$, as claimed.

So \eqref{lambdaSlambda} becomes $[\lambda,S^\lambda] = -\frac{3}{2}R$ as required.

\end{proof}

Since $R$ is skew-symmetric in its arguments, the first part of the theorem implies immediately
\begin{corollary}\label{SkilledbyV}
For any $X\in TTM'$, $S(V,X)=0$.
\end{corollary}

\section{Conformal transformations}\label{ConformalTransformationsSection}
A generalized conformal transformation is the transformation from the Lagrangian $G(x,v)$ to the Lagrangian $\widehat{G}(x,v)$ where
$$G(x,v) = \widehat{G}(x,v)J(x,v)^{-1}.$$
Here $J(x,v)$ is a non-zero function, smooth and defined in a neighborhood of the cone $G(x,v)=0$.  We have the homogeneities, valid for any real $t>0$:
\begin{align*}
G(x,tv) &= t^kG(x,v)\\
\widehat{G}(x,tv) &= t^p\widehat{G}(x,v)\\
J(x,tv) &= t^qJ(x,v)
\end{align*}
where
$$p-q=k\not=1,\quad p\not=1.$$

Denote by $\widehat{V}$ the dynamical vector field with respect to the transformed Lagrangian $\widehat{G}$.  

\begin{lemma}
The dynamical vector field $V$ transforms via $\widehat{V} = V + b H + GT$ where $b = -(p-1)^{-1}J^{-1}V(J) = (k-1)^{-1}J^{-1}\hat{V}(J)$ and $T$ is some vertical vector field.
\end{lemma}

\begin{proof}
Since both $V$ and $\widehat{V}$ satisfy Lemma \ref{spray}, $\lambda(V-\widehat{V})=0$, and therefore the two vector fields can only differ by a vertical vector field.  But working modulo $G=0$, since $\widehat{V}$ are dynamical vector fields on $\mathscr{H}$, $\widehat{V}$ must be in the span of $V$ and $H$, and so
$$\widehat{V} = V + bH\pmod{G=0}$$
for some function $b$.

It remains to show that $b = -(p-1)^{-1}J^{-1}V(J)=-(k-1)^{-1}J^{-1}\widehat{V}(J)$.  By definition of $\alpha$, $\widehat{\alpha}=J\alpha + G\,DJ$, and so pulling back to $\mathscr{H}$ gives
$$d\widehat{\alpha} \equiv dJ\wedge\alpha + J d\alpha\pmod{G,dG}.$$
Contracting with $V$ gives on the one hand $V\lrcorner d\widehat{\alpha} \equiv V(J) \alpha$ since $V\lrcorner\alpha=0$ and $V\lrcorner d\alpha = -(k-1)dG\equiv 0$. On the other hand, using $V=\widehat{V}-bH$ gives
$$V\lrcorner d\widehat{\alpha} \equiv -b(p-1)\widehat{\alpha}\equiv -b(p-1)J\alpha$$
because $\widehat{V}\lrcorner d\widehat{\alpha}=-(p-1)d\widehat{G}\equiv 0$ and $H\lrcorner d\widehat{\alpha} = (p-1)\widehat{\alpha}$.  Combining these two calculations gives $b = -(p-1)^{-1}J^{-1}V(J)$, as claimed.  By symmetry, $-b=-(k-1)^{-1}J\widehat{V}(J^{-1})=(k-1)^{-1}J^{-1}\widehat{V}(J)$, and so $b=-(k-1)^{-1}J^{-1}\widehat{V}(J^{-1})$ as well.
\end{proof}

\begin{lemma}\label{conformalchangeinconnection}
$\mathscr{L}_{\widehat{V}}\lambda = \mathscr{L}_V\lambda - b\lambda - Db\otimes H  - \alpha\otimes T + G\mathscr{L}_T\lambda.$
\end{lemma}
\begin{proof}
\begin{align*}
\mathscr{L}_{\widehat{V}}\lambda(X) &= [\widehat{V},\lambda X] - \lambda[\widehat{V},\lambda X]\\
&=[V + b H + GT,\lambda X] - \lambda[V + b H + GT,X]\\
&=\mathscr{L}_V\lambda(X) + b\mathscr{L}_H\lambda(X) - (D_Xb)H + G\mathscr{L}_T\lambda(X) - (D_XG)T + X(G)\lambda T\\
&=\left(\mathscr{L}_V\lambda - b\lambda - Db\otimes H + G\mathscr{L}_T\lambda - \alpha\otimes T\right)(X)
\end{align*}
where the last equality follows since $\mathscr{L}_H\lambda=-\lambda$, $D_XG=\alpha(X)$ by definition, and $\lambda T=0$ since $T$ is vertical.
\end{proof}

\subsection{Weyl tensor}\label{WeylSection}
Recall that $\Lambda_{\mathscr{H}}$ is the subbundle of $T\mathscr{H}$ consisting of vectors $X$ that annihilate $\alpha$: $\alpha(X)=0$.  Then $V\mathscr{H}\subset\Lambda_{\mathscr{H}}$ and also the dynamical vector field $V$ is a section of $\Lambda_{\mathscr{H}}$.  
\begin{definition}
The {\em umbral bundle} is the vector bundle $E$ over $\mathscr{H}$ defined as the quotient of $\Lambda_{\mathscr{H}}$ by the kernel of $g_h|_{\Lambda_{\mathscr{H}}\times\Lambda_{\mathscr{H}}}$.
\end{definition}
The umbral bundle is a rank $n-2$ vector bundle over $\mathscr{H}$.  It is so named because in \S \ref{RaychaudhuriSachsSection}, the pullback of $E$ along sections of $\mathscr{H}$ can be regarded as a space of infinitesimal screens onto which an object, placed into the null geodesic spray, will cast a shadow.  This interpretation is due to Sachs \cite{Sachs}.

The kernel of $g_h|_{\Lambda_{\mathscr{H}}\times \Lambda_{\mathscr{H}}}$ is the subspace of $\Lambda_{\mathscr{H}}$ spanned by $V\mathscr{H}$ and the dynamical vector field $V$.  Indeed, $g_h(V,X)=(k-1)\alpha(X)$, which vanishes if $X\in\Lambda_{\mathscr{H}}$.  Thus $\ker(g_h|_{\Lambda_{\mathscr{H}}\times \Lambda_{\mathscr{H}}})$ contains $V$.  It also contains $V\mathscr{H}$, since the kernel of the bilinear form $g_h$ on the full tangent space $TTM$ is $VTM$.  From rank considerations, $g_h|_{\Lambda_{\mathscr{H}}\times \Lambda_{\mathscr{H}}}$ has degree of degeneracy at most $n$, and so the kernel must in fact be equal to $V\mathscr{H}\oplus\operatorname{span}V$.  Thus
$$E = \frac{\Lambda_{\mathscr{H}}}{V\mathscr{H}\oplus\operatorname{span}V}.$$

If $X\in\Lambda_{\mathscr{H}}$, denote by $[X]$ the equivalence class of $X$ in $E$.  The tensor $g_h$ descends to a non-degenerate metric on $E$, via
$$g_E([X],[Y]) = g_h(X,Y).$$
When it is restricted to $\Lambda_{\mathscr{H}}\times\Lambda_{\mathscr{H}}$, $S$ vanishes if either argument is in the kernel of $g_h$ (by Corollary \ref{SkilledbyV}).  Thus by restriction $S$ defines a section of $E^*\otimes E^*$.  Define an endomorphism $S_E^\sharp:E\to E$ by setting $g_E(S_E^\sharp X,Y)=S(X,Y)$ for all $X,Y\in E$.

\begin{definition}
The {\em Weyl tensor} $W\in \Gamma_{\mathscr{H}}\left(E^*\otimes E^*\right)$ is the trace-free part of the restriction of the tidal force tensor $S(X,Y)$ to $(X,Y)\in E\times E$.  That is,
$$W(X,Y) = S(X,Y) - \left(\frac{1}{n-2}\tr S_E^\sharp\right)g_E(X,Y)$$
for $(X,Y)\in E\times E$.
\end{definition}

Notice that we are now working on $\mathscr{H}$ exclusively, and so $G=0$.

\begin{theorem}
The Weyl tensor depends only on $\mathscr{H}\subset TTM'$, not on the choice of defining function $G$.
\end{theorem}

In other words, the Weyl tensor is conformally invariant with respect to the class of conformal transformations described at the beginning of \S \ref{ConformalTransformationsSection}.

\begin{proof}
By Lemma \ref{tidalforcelemma}, $S^\sharp_E[X]=\frac{1}{2}\left[\overline{\lambda}^{-1}P(\mathscr{L}_V^2\lambda)X\right]$ for $X\in\Lambda_{\mathscr{H}}$.  So to prove that $W$ is conformally invariant, it is sufficient to compute $\widehat{P}(\mathscr{L}_{\widehat{V}}^2\lambda)$ on $E$, and then to neglect terms that are proportional to $\lambda$, since these will only modify the trace.  We shall therefore compute $\mathscr{L}_{\widehat{V}}^2\lambda$ modulo terms involving $G,dG,\alpha$, since these are zero on $E$, modulo $H$ since $\widehat{P}H=H$ which is in $VTM'$ and so also zero in $E$ in which the vertical space is quotiented, and modulo $V$ since $\widehat{P}V\equiv \widehat{P}\widehat{V}\pmod{H}\equiv 0$.  We treat each term of Lemma \ref{conformalchangeinconnection} in turn:
\begin{align*}
\left(\mathscr{L}_{\widehat{V}}\mathscr{L}_V\lambda\right)X &=  [\widehat{V},(\mathscr{L}_V\lambda) X] - (\mathscr{L}_V\lambda) [\widehat{V},X]\\
&=[V + b H + GT,(\mathscr{L}_V\lambda) X] - (\mathscr{L}_V\lambda)[V + b H + GT,X]\\
&\equiv\mathscr{L}^2_V\lambda(X) + b (\mathscr{L}_H\mathscr{L}_V\lambda)(X) + \mathscr{L}_{GT}(\mathscr{L}_V\lambda)(X)\pmod{G,dG,\alpha,H,V}\\
&\equiv\mathscr{L}^2_V\lambda(X) + \mathscr{L}_{GT}(\mathscr{L}_V\lambda)(X).
\end{align*}
by homogeneity of $V$ and $\lambda$.  Now $\mathscr{L}_{GT}(\mathscr{L}_V\lambda)(X)$ vanishes modulo $G$ for $X\in\ker dG$.  Hence
$$\left(\mathscr{L}_{\widehat{V}}\mathscr{L}_V\lambda\right)X\equiv \mathscr{L}^2_V\lambda(X).$$
The second term is
\begin{align*}
\left(\mathscr{L}_{\widehat{V}}(b\lambda)\right)X &\equiv \widehat{V}(b) \lambda X + b (\mathscr{L}_{\widehat{V}}\lambda)X\\
&\equiv  \widehat{V}(b) \lambda X +b(\mathscr{L}_V\lambda)X - b^2\lambda X. 
\end{align*}
The third term is
$$\left(\mathscr{L}_{\widehat{V}}(Db\otimes H)\right)X \equiv  -D_Xb\otimes V \equiv 0.$$
The remaining terms are zero, since they involve $\alpha$ and $G$:
\begin{align*}
\mathscr{L}_{\widehat{V}}(G\mathscr{L}_T\lambda) &\equiv 0\\
\left(\mathscr{L}_{\widehat{V}}(\alpha\otimes T)\right) &\equiv 0.
\end{align*}
Thus we have
$$\mathscr{L}_{\widehat{V}}^2\lambda \equiv \mathscr{L}_V^2\lambda  - b(\mathscr{L}_V\lambda) + (b^2- \widehat{V}(b))\lambda.$$
We now compute $\widehat{P}\mathscr{L}_{\widehat{V}}^2\lambda $.  By the transformation law for $\mathscr{L}_{\widehat{V}}\lambda$,
$$\widehat{P} = P - \frac{1}{2}(b\lambda + Db\otimes H +\alpha\otimes T - G\mathscr{L}_T\alpha).$$
Note that since $\mathscr{L}_v\lambda=2P-\operatorname{Id}$, $P\mathscr{L}_v\lambda = P$ and $\lambda\mathscr{L}_V\lambda = -\lambda$.  Among the remaining terms are those involving the $\alpha\otimes T$ term of $\widehat{P}$ contracted with a term of $\mathscr{L}_{\widehat{V}}^2\lambda$.  Of these, it follows from $\alpha\circ P=0$ that $\mathscr{L}_V\lambda$ preserves the annihilator of $\alpha$, and so the term $(\alpha\otimes T)(b\mathscr{L}_V\lambda)$ term vanishes when restricted to $\Lambda_{\mathscr{H}}$.  Since $\alpha\circ\lambda=0$, the $(\alpha\otimes T)\lambda$ terms vanish.   Finally, the term involving $(\alpha\otimes T)(\mathscr{L}^2_V\lambda)$ vanishes on $\Lambda_{\mathscr{H}}=\ker\alpha$ since the kernel of $\alpha$ is Lie derived along $V$.

So, applying $\widehat{P}$ to $\mathscr{L}_{\widehat{V}}^2\lambda$ gives
\begin{align*}
\widehat{P}\mathscr{L}_{\widehat{V}}^2\lambda &\equiv P\mathscr{L}_V^2\lambda - bP + (b^2-\widehat{V}(b))\lambda-\frac{1}{2}b\lambda\mathscr{L}_V^2\lambda-\frac{1}{2}b^2\lambda\\
&\equiv P\mathscr{L}_V^2\lambda  + \left(\tfrac{1}{2}b^2-\widehat{V}(b)\right)\lambda-\frac{1}{2}b\lambda\mathscr{L}_V^2\lambda - bP
\end{align*}
It is now sufficient to show that the last two terms cancel; that is:
$$\lambda\mathscr{L}_V^2\lambda = -2P.$$
We have
\begin{align*}
\lambda\mathscr{L}^2_V\lambda(X) &= \lambda\left([V,[V,\lambda X]]-2[V,\lambda[V,X]] + \lambda [V,[V,X]]\right) \\
&=\lambda[V,[V,\lambda X]] - 2\lambda[V,\lambda[V,X]]\\
&=\lambda[V,[V,\lambda X]] + 2\lambda[V,X].
\end{align*}
If $X$ is in the image of $P$, the first term vanishes because $\operatorname{im} P=\ker\lambda$, leaving only the second term which is $-2X$.  If instead $X\in\ker P$, then $\mathscr{L}_VX=-X$.  So
\begin{align*}
\lambda[V,[V,\lambda X]] + 2\lambda[V,X] &= \lambda[V,-X +\lambda[V,X]] - 2\lambda X \\
&=-\lambda [V,X] + \lambda[V,-\lambda X] -2\lambda X\\
&=2\lambda X -2\lambda X =0
\end{align*}
as required.  Thus, in summary
$$\widehat{P}\mathscr{L}_{\widehat{V}}^2\lambda \equiv P\mathscr{L}_V^2\lambda  + \left(\tfrac{1}{2}b^2-\widehat{V}(b)\right)\lambda.$$
Since the term multiplying $\lambda$ only modifies the trace of $S$, this completes the proof.
\end{proof}

\section{Raychaudhuri--Sachs equations}\label{RaychaudhuriSachsSection}
We review the Raychaudhuri--Sachs equation of standard general relativity. Let $M$ be a spacetime manifold of dimension $n\ge 3$, equipped with an indefinite metric, $g$ of signature $(p, q)$.  Denote the Levi-Civita connection of $g$ by $\nabla$.

Let $k$ be a null vector field that is nowhere zero and satisfies the equation of an affinely parametrized geodesic $\nabla_kk=0$.  The integral curves of $k$ are null geodesics that foliate $M$: that is, they constitute a null geodesic congruence.  Associated to the vector field $k$ is a natural vector bundle $K$ of dimension $n-2$ with a metric of signature $(p-1,q-1)$ (so Euclidean in the case where $g$ is Lorentzian).  This bundle consists of the $(n-2)$-plane elements (or ``screens'') onto which the infinitely near curves of the congruence would cast the shadow of an object.  This bundle was introduced in this way by Sachs \cite{Sachs}.  A precise definition of this bundle is in section \ref{BundleKSection}.

The Raychaudhuri--Sachs equation then governs the rate at which this shadow expands (or contracts) as the screens advance along a particular geodesic of the congruence.  A principal ingredient in the derivation of the equation is the notion of the divergence of $k$, of which there are potentially several candidates (that turn out to agree):
\begin{itemize}
\item The Lie derivative of the volume element of $M$ along $k$.
\item The Lie derivative of a natural volume element for the bundle $K$.
\item The trace of the endomorphism $\nabla k$.
\end{itemize}

\subsection{Notation and conventions}
The curvature tensor $R\in\Gamma_M(\wedge^2T^*M\otimes \End(TM))$ is defined by the relation
$$R(X,Y)Z = (\nabla_X\nabla_Y - \nabla_Y\nabla_X-\nabla_{[X,Y]})Z.$$
The Ricci tensor is given by
$$\operatorname{Ric}(X,Y) = \operatorname{tr}(Z\mapsto R(X,Z)Y).$$
The metric defines an isomorphism between the tangent and cotangent bundles of $M$: define $g:TM\to T^*M$ by
$$g(X) : Y\mapsto g(X,Y).$$
This is a self-adjoint transformation (by the symmetry of $g$) that is invertible (by the non-degeneracy of $g$).  The inverse $g^{-1}:T^*M\to TM$ defines a metric $g^{-1}$ on $T^*M$ by
$$g^{-1}(\alpha,\beta) = \beta(g^{-1}(\alpha)).$$
The metric will be used to convert vectors into covectors systematically using the ``musical isomorphism'':
\begin{itemize}
\item If $X$ is a vector, define $X^\flat=g(X)$.
\item If $\alpha$ is a covector, define $\alpha^\sharp=g^{-1}(\alpha)$.
\end{itemize}

The volume element of $M$ is a density on $M$ that is defined on a collection of vectors $v_1,\dots,v_n$ by
$$|\Omega(v_1,\dots,v_n)|^2 =|\det [g(v_i,v_j)]_{i,j=1,\dots,n}|.$$
This is a section of the density bundle $|\wedge^nT^*M|$.  If an orientation is given on $M$, then it is possible to choose a representative volume form, denoted $\Omega$, for the density $|\Omega|$.  In a distinguished oriented local orthonormal basis\footnote{For a metric of indefinite signature, an orthonormal basis is any basis such that $g(X_i,X_j)=\pm\delta_{ij}$.} of vector fields $X_1,\dots,X_n$, 
$$\Omega(X_1,\dots,X_n) = |\Omega(X_1,\dots,X_n)|=1,$$
extended by multilinearity.

Since a local one-parameter group of diffeomorphisms must preserve orientation, the Lie derivative of $|\Omega|$ along any differentiable vector field is well-defined.  The divergence of a differentiable vector field $X$ is defined by
$$(\div X)|\Omega| = \mathscr{L}_X|\Omega|.$$

Now, for the Raychaudhuri--Sachs equations, assume in addition that $k$ is hypersurface orthogonal.  This is equivalent to the condition that the distribution $k^\perp = (k^\flat)^0\subset TM$ of $(n-1)$-planes annihilating $k^\flat$ be integrable in the sense of Frobenius: $k^\flat\wedge dk^\flat=0$.

\subsection{The umbral bundle of the congruence}\label{BundleKSection}
Let $k^\perp$ denote the distribution of $(n-1)$-planes orthogonal to $k$.  Thus, at a point $x\in M$,
$$k^\perp_x = \{v\in T_xM\mid g(k,v)=0\} = (k^\flat_x)^0.$$
\begin{lemma}\label{Liek}
The distribution $k^\perp$ is Lie-derived along $k$.  That is, if $v\in\Gamma_M(k^\perp)$, then $\mathscr{L}_kv\in\Gamma_M(k^\perp)$.
\end{lemma}
\begin{proof}
If $v$ is a section of $k^\perp$, then $g(k,v)=0$.  So
\begin{align*}
0 &= k(g(k,v)) = g(\nabla_kk,v) + g(k,\nabla_kv) = g(k,\nabla_k v)\\
&=g(k,\mathscr{L}_kv) + g(k,\nabla_vk)\\
&=g(k,\mathscr{L}_kv)  + \frac{1}{2}v(g(k,k)) = g(k,\mathscr{L}_kv)  
\end{align*}
so $\mathscr{L}_kv$ is also in $k^\perp$, as required.
\end{proof}

Note that $k$ is a section of $k^\perp$, since it is null.  Therefore the following definition makes sense:
\begin{definition}
Let $K$ be the quotient bundle $K=k^\perp/\operatorname{span} k$.
\end{definition}
If a small object is placed in the path of the congruence $k$, then the bundle $K$ naturally describes a family of screens onto which the shadow of an object is cast.  Hence, this is the {\em umbral bundle} for the null geodesic congruence $k$.  It is the pullback of the umbral bundle defined in \S \ref{WeylSection} by the section $k$ of the null cone bundle $\mathscr{H}$; see \S \ref{GeometricRaychaudhuriSection} for more details.

Let $[v]$ denote the equivalence class of $v\in k^\perp$ modulo $k$.  Since $k^\perp$ and $\operatorname{span}k$ are both Lie derived along $k$, the Lie derivative $\mathscr{L}_k$ descends to a differential operator on the quotient $K$, by setting
$$\mathscr{L}_k[v] = [\mathscr{L}_kv].$$
The Lie derivative extends to a unique derivation on the tensor algebra of $K$ that commutes with tensor contraction.

The metric $g$ in $TM$ induces a bilinear form $g_{k^\perp}$ on $k^\perp$, and the vector $k$ is in the kernel of $g_{k^\perp}$.  Hence $g_{k^\perp}$ descends to a bilinear form on $K$ via the rule
$$g_K([X],[Y]) = g_{k^\perp}(X,Y).$$
The bilinear form $g_K$ is a metric of signature $(p-1,q-1)$ on $K$.

The tidal force along $k$ is the endomorphism $S^\sharp:TM\to TM$ given on vectors $X$ by
$$S^\sharp X = R(k,X)k.$$
Since $S^\sharp k=0$ and the image of $S^\sharp$ is orthogonal to $k$, $S^\sharp$ induces an endomorphism of $K$ via
$$S^\sharp_K[X] = [S^\sharp X].$$
The bilinear form $S$ on $TM$ and $S_K$ on $K$ given by
$$S(X,Y) = g(S^\sharp X,Y),\quad S_K([X],[Y])=g_K(S^\sharp_K[X],[Y])$$
are both symmetric, by the symmetries of the Riemann tensor.

\subsection{Divergence}\label{DivergenceSection}
\begin{definition}
Let $X$ be a vector field.  The divergence of $X$, denoted $\div X$, is defined by the equation 
$$(\div X)|\Omega| = \mathscr{L}_X|\Omega|$$
\end{definition}

For the vector field $X$, define the endomorphism $\nabla X$ of $TM$ by $\nabla X:Y\mapsto \nabla_YX$.
\begin{lemma}
The divergence of $X$ is the trace of $\nabla X\in\Gamma_M(\End(TM))$
$$\div X = \tr \nabla X.$$
\end{lemma}

\begin{proof}
Let $v_1,\dots,v_n$ be a local basis of smooth sections of $TM$,  and let $\alpha^1,\dots,\alpha^n$ be the dual basis of $T^*M$, defined by $\alpha^i(v_j)=\delta^i_j$. Let $\Omega$ be the local section of $\wedge^nT^*M$ representing $|\Omega|$ obtained by declaring the basis $v_i$ to be positively oriented.  First note that if $\alpha$ is a one-form and $Y$ a tangent vector, then
$$0=Y\lrcorner\,(\alpha\wedge\Omega) = \alpha(Y)\Omega-\alpha\wedge Y\lrcorner\,\Omega$$
so
\begin{equation}\label{wedgelrcorner}
\alpha\wedge Y\lrcorner\,\Omega = \alpha(Y)\Omega.
\end{equation}

By Cartan's identities,
\begin{align*}
\mathscr{L}_X\Omega &= d(X\lrcorner\,\Omega) = \sum_i \alpha^i\wedge \nabla_{v_i} (X\lrcorner\,\Omega)\\
&= \sum_i \alpha^i\wedge (\nabla_{v_i} X)\lrcorner\,\Omega + \sum_i \alpha^i \wedge X\lrcorner\,\nabla_{v_i}\Omega\\
&= \sum_i \alpha^i \wedge(\nabla_{v_i} X)\lrcorner\,\Omega\\
&=\sum_i \alpha^i(\nabla_{v_i} X)\Omega = \tr(\nabla X)\Omega
\end{align*}
by \eqref{wedgelrcorner}.
\end{proof}

Fixing an orientation on $TM$ equips the bundle $K$ with an induced orientation, and the associated volume forms are related by
$$k\lrcorner\, \Omega = k^\flat \wedge \Omega_K$$
where $\lrcorner$ is the interior product.  The validity of this equation does not depend on the choice of coset representative of $\Omega_K$ modulo the ideal generated by $k^\flat$, and so defines $\Omega_K$ uniquely as a section of $\wedge^{n-2}K^*$.  If no orientation on $M$ is specified, then this only defines a density $|\Omega_K|$ in the determinant bundle $|\wedge^{n-2}K^*|$.

\begin{definition}
Define $\div_Kk$ by
$$(\div_Kk)|\Omega_K| = \mathscr{L}_k|\Omega_K|.$$
\end{definition}

\begin{lemma}
$\div_Kk=\div k$
\end{lemma}
\begin{proof}
Working locally with an orientation on $M$, we have $k\lrcorner\,\Omega = k\wedge\Omega_K$.  So
\begin{align*}
\mathscr{L}_k(k\lrcorner\,\Omega) &= \mathscr{L}_k(k\wedge\Omega_K)\\
k\lrcorner\,\mathscr{L}_k\Omega &= k\wedge\mathscr{L}_k\Omega_K\\
(\div k) k\lrcorner\,\Omega &= (\div_K k)k\wedge\Omega_K.
\end{align*}
\end{proof}

The image of the endomorphism $\nabla k$ lies in $k^\perp$, since $g(\nabla_Xk,k)=\frac{1}{2}X(g(k,k))=0$.  Furthermore, $k$ lies in the kernel of $\nabla k$, since $k$ is an affinely parametrized geodesic vector field.  Therefore, $\nabla k$ descends to an endomorphism $\nabla k|_K : K\to K$.

\begin{lemma}\label{trnabla}
For any $p=1,2,\dots$,
$$\tr((\nabla k)^p) = \tr((\nabla k|_K)^p)$$
\end{lemma}
\begin{proof}
In general, if $A$ is an endomorphism of a vector space $V$ whose image lies in a subspace $W$, then $\tr A = \tr (A|_W)$.  Since $\nabla k$ is a linear operator whose image lies in $k^\perp$, $\tr(\nabla k)^p= \tr((\nabla k)^p|_{k^\perp})$.  But $(\nabla k)^p|_{k^\perp} = (\nabla k|_{k^\perp})^p$, and so $\tr(\nabla k)^p = \tr(\nabla k|_{k^\perp})^p$.  Now, since $\operatorname{span} k$ lies in the kernel of $\nabla k|_{k^\perp}$, $\tr(\nabla k|_{k^\perp})^p = \tr(\nabla k|_K)^p$, as required.
\end{proof}

\subsection{Rate of change of the divergence}
The purpose of this section is to compute the rate of change of the divergence of $k$.  Let $R(k,-)k$ denote the endomorphism $R(k,-)k:X\mapsto R(k,X)k$.  Then:

\begin{lemma}
$\nabla_k \nabla k = -(\nabla k)^2  + R(k,-)k$
\end{lemma}

\begin{proof}
For a vector field $X$,
\begin{align*}
(\nabla_k \nabla k)(X) &= \nabla_k \nabla_X k - \nabla_{\nabla_k X}k \\
&=R(k,X)k + \nabla_X\nabla_kk+\nabla_{[k,X]}k - \nabla_{\nabla_kX}k\\
&=R(k,X)k - \nabla_{\nabla_Xk}k\\
&= [-(\nabla k)^2 + R(k,-)k](X)
\end{align*}
\end{proof}

\begin{lemma}\label{Raychaudhuri1}
\begin{align*}
k(\div k) &= -\tr[(\nabla k)^2] +\operatorname{Ric}(k,k)\\
&= -\tr[(\nabla k|_K)^2] +\operatorname{Ric}(k,k)
\end{align*}
\end{lemma}
\begin{proof}
The first equation follows by taking a trace from the previous lemma.  The second equation follows from $\tr(\nabla k)^2=\tr(\nabla k|_K)^2$.
\end{proof}

\begin{lemma}\label{trS}
$\tr S^\sharp = \tr S_K^\sharp$
\end{lemma}
\begin{proof}
The image of $S^\sharp$ lies in $k^\perp$ and the kernel of $S^\sharp$ contains $k$.  Thus the lemma follows by the argument of Lemma \ref{trnabla}.
\end{proof}

\subsection{Invariant decomposition}
Let
$$\nabla k|_K = \operatorname{Alt}(\nabla k|_K) + \operatorname{Sym}_0(\nabla k|_K) + \frac{1}{n-2}(\div k)\operatorname{Id}_K$$
be the decomposition of $\nabla k|_K$ into its irreducible components for the action of $O(p-1,q-1)$: the alternating, symmetric trace-free, and trace parts.  Here the metric $g_K$ is used to identify $\operatorname{End}(K)$ with $K^*\otimes K^*$ in order to define the symmetric and alternating parts.

For the next theorem, introduce the following notation, standard in the relativity literature when $n=4$:
\begin{itemize}
\item $\theta=\div k$ is called the {\em expansion} of the congruence $k$ in the relativity literature
\item $\sigma =\operatorname{Sym}_0(\nabla k|_K)$ is the {\em shear tensor}
\item $\rho= \operatorname{Alt}(\nabla k|_K)$ is the {\em rotation tensor}
\end{itemize}

\begin{theorem}\label{Raychaudhuri2}
\begin{align*}
k(\theta) &= -\tr(\rho^2)-\tr(\sigma^2) - \frac{\theta^2}{n-2} + \operatorname{Ric}(k,k)\\
&= -\tr(\rho^2)-\tr(\sigma^2) - \frac{\theta^2}{n-2} + \tr S^\sharp.
\end{align*}
\end{theorem}
\begin{proof}
This is a restatement of Lemma \ref{Raychaudhuri1} under the decomposition
$$\nabla k|_K = \rho + \sigma + \frac{1}{n-2}\theta\operatorname{Id}_K.$$
The absence of cross-terms owes to the orthogonality of the different irreducible representations of $O(p-1,q-1)$.  The second equality follows from the definition of $S^\sharp$.
\end{proof}

\subsection{Hypersurface orthogonality}
If $k$ is hypersurface orthogonal, then the distribution $k^\perp=(k^\flat)^0$ is integrable in the sense of Frobenius, and therefore $dk^\flat \equiv 0\pmod{k^\flat}$.

\begin{lemma}
If $k$ is hypersurface orthogonal, then $\tr(\rho^2)=0$.
\end{lemma}
\begin{proof}
If $k$ is hypersurface orthogonal, then there exists locally a one-form $\mu$ such that $dk^\flat=\mu\wedge k^\flat$.  Since $k$ is a geodesic vector field, $k\lrcorner dk^\flat = 0$, and since $k$ is also null $k^\flat(k)=0$, so $\mu(k)=0$ as well.  Now
$$\tr(\rho^2) = dk^\flat(k,\mu^\sharp) = 0$$
as claimed.
\end{proof}

Theorem \ref{Raychaudhuri2} becomes the {\em Raychaudhuri--Sachs equations}:
\begin{corollary}\label{Raychaudhuri3}
If $k$ is hypersurface orthogonal, then
$$k(\theta) = -\tr(\sigma^2) - \frac{\theta^2}{n-2} + \tr S^\sharp.$$
\end{corollary}

\subsection{Raychaudhuri effect}
\begin{lemma}
Suppose that $J$ is a vector field that Lie commutes with $k$.  Then $J$ is a Jacobi field along any integral curve of $k$.
\end{lemma}
\begin{proof}
Covariantly differentiating $0=[k,J] = \nabla_kJ-\nabla_J k$ along $k$ gives
\begin{align*}
0&=\nabla_k^2J - \nabla_k\nabla_Jk\\
&=\nabla_k^2J - R(k,J)k
\end{align*}
which is the Jacobi equation
\end{proof}
In particular, since $k$ is hypersurface orthogonal, there are $n-2$ (Jacobi) vector fields $J_1,\dots,J_{n-2}$ that are orthgononal to $k$, Lie commute with $k$, and are linearly independent of $k$.  On passing to the quotient, these Jacobi fields define a basis of $K$.  Pick such a basis, and let $\lambda_K=|\Omega_K(J_1,\dots,J_{n-2})|$. 

In the Lorentzian case of a space-time of $n$-dimensions, the signature of the metric $g_E$ of the bundle $E$ is either positive or negative definite, according as $g$ has signature $(n-1,1)$ or $(1,n-1)$.  Thus in the Raychaudhuri--Sachs equations, the trace $\tr(\sigma^2)$ is non-negative, and it is zero if and only if $\sigma=0$.  Thus Corollary \ref{Raychaudhuri3} gives
$$k(\theta) \le \tr S^\sharp,$$
or equivalently,
$$\mathscr{L}_k^2\lambda_K \le \tr S^\sharp\lambda_K.$$
The null positive energy condition is the condition
$$\operatorname{Ric}(n,n)\le 0 \quad\text{for all null vectors $n$.}$$
So when the null positive energy condition holds,
$$\mathscr{L}_k^2\lambda_K \le 0$$
Note that this equality only requires that  $\operatorname{Ric}(k,k)\le 0$ be valid for the particular tangent vectors along the given null geodesic.

Now suppose that $\theta < 0$ at some point $x_0$ of the congruence.  By definition of $\theta$, at that point $\mathscr{L}_k\lambda_K = \theta\lambda_K < 0.$  Then $\lambda_K$ will become zero along the geodesic tangent to $k$ through $x_0$ at some time prior to the finite affine parameter $t=-(n-2)/\theta(x_0)$.  Since $k$ is hypersurface orthogonal, the vectors $J_1,\dots,J_{n-2}$ span the tangent space of this hypersurface up to the point where the volume $\lambda_K$ degenerates to zero.  At or before that point, the geodesic in question must have a conjugate point.  The existence of this conjugate point is the key to the proof by Sir Roger Penrose \cite{PenroseTDGR} of his singularity theorem.

\section{The geometric Raychaudhuri--Sachs theorem}\label{GeometricRaychaudhuriSection}
In this section, we lift the geometry underlying the Raychaudhuri--Sachs theorem to the bundle $\mathscr{H}$ and at the same time generalize it to regular causal geometries.   We first recall some basic sheaf theory.

Let $p:\mathbb{Y} \rightarrow \mathbb{X}$ be a (continuous) map of topological spaces. 
\begin{itemize}\item  A point $y \in \mathbb{Y}$ is said to be a sheaf point if and only if there exists an open set $U_y$ in $\mathbb{Y}$, such that $y \in U_y$ and such that the restriction of $p$ to $U_y$ is a homeomorphism onto its open image $p(U_y)$ in $\mathbb{X}$. 
\item   The sheaf space $\mathcal{S}_p\subset \mathbb{Y}$ of $p$ is the collection of all its sheaf points, with the induced topology.   Note that $\mathcal{S}_p$ is an open subset of $\mathbb{Y}$. 
\item The triple $(\mathbb{Y}, \mathbb{X}, p)$ is said to be a sheaf if and only if $p$ is surjective and $\mathcal{S}_p = \mathbb{Y}$.
\item The triple $(\mathbb{Y}, \mathbb{X}, p)$ is said to be a stack if and only if $p$ is surjective and $\mathcal{S}_p$ is dense in $\mathbb{Y}$, i.e. the closure $\overline{\mathcal{S}_p} = Y$.   
\item   The triple $(\mathbb{Y}, \mathbb{X}, p)$ is said to be a branched cover if and only if it is a stack and both $\mathbb{Y}$ and $\mathbb{X}$ are Hausdorff topological spaces. \end{itemize}
For example:
\begin{itemize}\item  Put $\mathbb{S}^1 = \{ (x, y) \in \mathbb{R}^2: x^2 + y^2 = 1\}$, the unit circle in the plane.   Then the map $e: \mathbb{R} \rightarrow \mathbb{S}^1$ given by the formula $e(t) = (\cos(t), \sin(t))$, for any $t \in \mathbb{R}$ makes $(\mathbb{R}, \mathbb{S}^1, e)$ a sheaf.  
\item Consider the complex parabola $\mathbb{Y} = \{(x, y) \in \mathbb{C}^2: y^2 = x\}$ and let $p(x, y) = x \in \mathbb{C}$ for any $(x, y) \in \mathbb{Y}$.  Then the triple $(\mathbb{Y}, \mathbb{C}, p)$ is a stack, with $\mathcal{S}_p = \mathbb{Y} - \{(0, 0)\}$ and is a branched cover.
\end{itemize}
If $(\mathbb{Y}, \mathbb{X}, p)$ is a stack, and if $U$ is an open subset of $\mathbb{X}$, then a section $s$ of the stack over $U$ is a map $s: U \rightarrow \mathbb{X}$, such that $p\circ s = id_{U}$.

The key concept we need is that of a Sachs manifold.   Let  $M$ be a manifold of dimension $n$ and let $\mathscr{G}\subset \mathbb{SM}$ be a regular causal geometry.  Denote the natural surjection from $\mathscr{G}$ to $M$ by $p$.  Denote a representative one-form of the contact structure of $\mathscr{G}$ by $\alpha_{\mathscr{G}}$.  Put $\mathscr{H} = \sigma^{-1}(\mathscr{G})$ and denote by $\alpha_{\mathscr{H}}$ a representative one-form of the contact structure of $\mathscr{H}$.  Also denote by $q$ the natural surjection from $\mathscr{H}$ to $M$.
\begin{itemize}\item A Sachs manifold for the causal geometry $\mathscr{G}$ is a smooth submanifold $\mathcal{S}$ of $\mathscr{G}$ of dimension $n$, such that:
\begin{itemize}\item The triple $(\mathcal{S}, M, p|_{\mathcal{S}})$ is a branched cover. \item $\mathcal{S}$ is ruled by (unparametrized) null geodesics: i.e. the null geodesic spray $V$ of $\mathscr{G}$ is everywhere tangent to $\mathcal{S}$.
\item $\mathcal{S}$ is hypersurface orthogonal:  the restriction of the three-form $\alpha_{\mathscr{G}}d \alpha_{\mathscr{G}}$ to $\mathcal{S}$ vanishes identically.
\end{itemize}
\item An affine Sachs manifold for the causal geometry $\mathscr{G}$ is a submanifold $\mathcal{T}$ of $\mathscr{H}$ of dimension $n$, such that:
\begin{itemize}\item The triple $(\mathcal{T}, M, q|_{\mathcal{T}})$ is a branched cover. \item $\mathcal{T}$ is ruled by affinely parametrized null geodesics: i.e. the null geodesic spray $V$ of $\mathscr{H}$ is everywhere tangent to $\mathcal{T}$.
\item $\mathcal{T}$ is hypersurface orthogonal:  the restriction of the three-form $\alpha_{\mathscr{H}}d \alpha_{\mathscr{H}}$ to $\mathcal{T}$ vanishes identically.
\end{itemize}
\item A Sachs section for a causal geometry $\mathscr{H}$ over an open set $U\subset M$ is a section of a given Sachs manifold, whose domain is $U$.
\item A Sachs congruence on an open subset $U\subset M$ is the foliation of $U$ by the null geodesics giving the foliation of a Sachs section.  Note that the Sachs congruence is automatically hypersurface orthogonal, with normals the (null) tangent vectors to the foliation and the congruence and the section determine each other uniquely.
\end{itemize}
In the special case of a standard space-time, the Sachs congruence exactly agrees with the congruence needed for the Raychaudhuri--Sachs equation and we see that in that case the affine Sachs manifold is simply the natural lift to the tangent bundle of the Sachs congruence, so we have a natural generalization.

Now let a Sachs section $s: M \rightarrow \mathscr{H}$ be given.  The connection $P:T\mathscr{H}\to V\mathscr{H}$ defines an endomorphism $P_s\in\operatorname{End}(TM)$ given by
$$P_s(X) = \overline{\lambda}^{-1}P(s_* X).$$
Let $k$ be the tangent vector field of the congruence, so $s_*k=V$.  The tensor $g_h$ pulls back under $s$ to a metric $g_s=s^*g_h$ on $M$.  Moreover, $k^\flat=s^*\alpha$.  Since $g_h(V,V)=k(k-1)G$, it follows that $g_s(k,k)=0$ since $s$ is a section of $\mathscr{H}$ where $G=0$.

The bundle $K$ is defined as before as $k^\perp/\operatorname{span} k$, where $k^\perp$ is the orthogonal complement of $k$ with respect to the metric $g_s$.  This is naturally isomorphic to the pullback under $s$ of the umbral bundle $E$ defined in section \ref{WeylSection}.  The metric $g_s$ induces a metric $g_K$ on $K$, which is of definite signature if $g_v$ has Lorentzian signature.  Let $\nabla$ denote the Levi-Civita connection of $g_s$.  The Lie derivative $\mathscr{L}_k$ preserves $\ker k^\flat$, by Lemma \ref{Liek}.  Likewise the Lie derivative extends to all associated tensor bundles.

\begin{lemma}
$P_s = \nabla k$ where $\nabla$ is the Levi-Civita connection associated with the metric $g_s$.  In particular $k$ is an affinely parametrized geodesic with respect to the connection $\nabla$.  Moreover, the pullback of the tidal force tensor along $s$ is the sectional curvature of $\nabla$ in the direction of $k$:
$$S(s_*X,s_*Y) = g_s(R(k,X)k,Y)$$
where $R$ is the Riemann tensor associated to $\nabla$.
\end{lemma}

\begin{proof}
The proof of the first claim proceeds by verifying that the two tensors have the same skew and symmetric parts.  On the one hand,
\begin{align*}
(\mathscr{L}_kg_s)(X,Y) &= k(g_s(X,Y)) - g_s(\nabla_kX-\nabla_Xk,Y) -g_s(X,\nabla_kY-\nabla_Yk)\\
&= g_s(\nabla_kX,Y) + g_s(X,\nabla_kY)- g_s(\nabla_kX-\nabla_Xk,Y) -g_s(X,\nabla_kY-\nabla_Yk)\\
&=g_s(\nabla_Xk,Y) + g_s(X,\nabla_Yk).
\end{align*}
On the other hand,
\begin{align*}
(\mathscr{L}_kg_s)(X,Y) &= (\mathscr{L}_Vg_h)(s_*X,s_*Y) = g_v(Ps_*X,\lambda s_*Y)+g_v(\lambda s_*X,Ps_*Y)\\
&=g_h(\overline{\lambda}^{-1}Ps_*X, s_*Y)+g_v(s_*X,\overline{\lambda}^{-1}Ps_*Y)\\
&=g_s(P_sX,Y)+g_s(X,P_sY).
\end{align*}
This shows that $\nabla k$ and $P_s$ have the same symmetric part.

For the skew part, on the one hand
\begin{align*}
2(s^*d\alpha)(X,Y) &= 2dk^\flat(X,Y)\\
&= g_s(\nabla_Xk,Y) - g_s(X,\nabla_Yk)
\end{align*}
and on the other hand
\begin{align*}
2(s^*d\alpha)(X,Y) &= 2d\alpha(s_*X,s_*Y) =  g_v(Ps_*X,\lambda s_*Y) - g_v(\lambda s_*X,Ps_*Y)\\
&= g_h(\overline{\lambda}^{-1}Ps_*X,s_*Y) - g_h(s_*X,\overline{\lambda}^{-1}Ps_*Y)\\
&=g_s(P_sX,Y) - g_s(X,P_sY).
\end{align*}

Since $P_s=\nabla k$, $\nabla_kk=P_sk=\overline{\lambda}^{-1}PV=0$ since $V$ is horizontal for the Ehresmann connection $P$.  Hence $k$ is an affinely parametrized geodesic.

For the final claim, Theorem \ref{tidalforce} implies that it is sufficient to prove
$$\frac{1}{2}(\mathscr{L}_k^2g_s)(X,Y) = g_s(\nabla_Xk,\nabla_Yk) - g_s(R(k,X)k,Y)$$
since $P_s=\nabla k$ by the first part of the lemma.   The identity
$$\mathscr{L}_k\nabla k = \nabla_k\nabla k$$
holds, so
\begin{align*}
\frac{1}{2}(\mathscr{L}_k^2g_s)(X,Y)&= g_s(\nabla_Xk,\nabla_Yk) + \frac{1}{2}g_s((\nabla_k\nabla k)(X),Y) + \frac{1}{2}g_s(X,(\nabla_k\nabla k)(Y))+\\
&\qquad\quad+\frac{1}{2}g_s(\nabla_{\nabla_X k} k,Y)+\frac{1}{2}g_s(X,\nabla_{\nabla_Y k} k)\\
&=g_s(\nabla_Xk,\nabla_Yk) + \frac{1}{2}\left(g_s(R(k,X)k,Y) + g_s(X,R(k,Y)k)\right)\\
&=g_s(\nabla_Xk,\nabla_Yk) + g_s(R(k,X)k,Y)
\end{align*}
by the symmetries of the Riemann tensor.
\end{proof}

The operator $S^\sharp_E:\pi^{-1}_{TM'}TM\to \pi^{-1}_{TM'}TM$ defined in \S \ref{WeylSection}, when restricted to the section $s$ defines an operator $S^\sharp_s:TM\to TM$.  By the previous lemma, $S^\sharp_s(X) = R_s(k,X)k$.  Moreover, as in \S \ref{RaychaudhuriSachsSection}, the image of $S^\sharp_s$ lies in $k^\perp$ and its kernel contains $k$, so $S^\sharp_s$ descends to and operator $S^\sharp_K:K\to K$.  Moreover, $\tr S^\sharp_s=\tr S^\sharp_K = \operatorname{Ric}(k,k)$

As in \S \ref{DivergenceSection}, the divergence of $k$ can be defined in several equivalent ways.  If $|\Omega|$ is the canonical density associated to the metric $g_s$, then
$$\mathscr{L}_k|\Omega| = (\div k)|\Omega|.$$
If $|\Omega_E|$ is the canonical section of the determinant line bundle $|\wedge^{n-2}E|$, then
$$\mathscr{L}_k|\Omega_E| = (\div_E k)|\Omega_E|.$$
Alternatively, the divergence can be defined as the trace of $\nabla k=P_s$, or the trace of $\nabla k|_E=P_s|_E$.  The results of \S \ref{DivergenceSection} imply that these are equal:
\begin{lemma}
$\theta=\div k = \div_E k = \tr(\nabla k) = \tr (P_s) = \tr(\nabla_k|_E) = \tr(P_s|_E)$
\end{lemma}

The proof of Theorem \ref{Raychaudhuri2} goes through as in \S \ref{RaychaudhuriSachsSection}:

\begin{theorem}
Let
\begin{align*}
P_s|_E&=\operatorname{Alt}P_s|_E + \operatorname{Sym}_0P_s|_E + \frac{\tr P_s|_E}{n-2}\operatorname{Id}_E\\
&=\rho + \sigma + \frac{\theta}{n-2}\operatorname{Id}_E
\end{align*}
be the decomposition of $P_s$ into its irreducible $O(p-1,q-1)$ components.  Then
$$k(\theta) = -\tr(\rho^2)-\tr(\sigma^2) - \frac{\theta^2}{n-2} + \tr S^\sharp.$$
\end{theorem}
 
\subsection{The Lorentzian case: the geometric Raychaudhuri--Sachs effect}
Now consider the case that the fibre metric $g_v$ in $VTM'$ is Lorentzian, which implies in turn that $g_s$ is also Lorentzian, and so the metric $g_K$ of the bundle $K$ has positive or negative definite signature.  Then the quantity $\tr(\sigma^2)$ of the Raychaudhuri--Sachs equation is non-negative.   Also impose the positive energy condition: $\tr S^\sharp \le 0$.  As in \S \ref{RaychaudhuriSachsSection}, let $J_1,\dots, J_{n-2}$ be a collection of vector fields orthgonal to $k$ that commute with $k$, and set $\lambda_K=|\Omega_K(J_1,\dots,J_{n-2})|$.  Then
$$\mathscr{L}_k\lambda_K = \theta\lambda_K, \hspace{10pt} \mathscr{L}^2_k \lambda_K  \le 0.$$
Now if at a point of the congruence we have $\theta < 0$, then it follows that the graph of $\lambda_K$ along the (affinely parametrized) null geodesic through the point is decreasing and concave down, so $\lambda_K$ reaches zero in finite affine parameter time in the future.  So we have the theorem:

\begin{theorem}
Let $X$ be a given null geodesic in $M$ that is future complete, so its affine parameter ranges to positive infinity. Suppose that everywhere along $X$ the positive energy condition $\tr S^\sharp \le 0$ holds.   Suppose there is a section of a Sachs manifold, defined in a neighborhood of $X$, such that $X$ is a member of the congruence foliating the Sachs manifold.  Then the divergence of the congruence is everywhere non-negative along $X$.
\end{theorem}

\appendix

\section{Notational conventions}
\begin{itemize}
\item If $M$ is a manifold and $E\to M$ is a bundle, then the projection is denoted by $\pi_E$.  The space of smooth sections is denoted by $\Gamma_M(E)$.
\item If $A$ is a vector space and $S\subset A^*$ is a subset, then the annihilator of $S$, denoted by $S^0$, is defined by
$$S^0 = \{x\in A|\alpha(x)=0 \text{\ for all $\alpha\in S$}\}. $$
\item Arbitrary vector fields are denoted by uppercase latin letters at the end of the alphabet: $W,X,Y,Z$.  In Section \ref{RaychaudhuriSachsSection}, $k$ is used to denote a null geodesic vector field.  Throughout the paper, $V$ denotes the null geodesic spray and $H$ the homogeneity operator.
\end{itemize}

\section{Coordinate calculations}
\subsection{The Hamiltonian-Lagrangian approach to the dynamics}
Let $\mathscr{H}$ have the local defining equation $G(x, v) = 0$.  Here the smooth function $G(x, v)$ is defined over an open set $\mathbb{U}$ of $\mathbb{TM}'$, which is invariant under scaling: i.e. $(x, v) \in \mathbb{U}$ implies that $(x, tv) \in \mathbb{U}$, for any positive real $t$.  Further we may take the function $G(x, v)$ to be homogeneous: $G(x, tv) = t^{k} G(x, v)$, for some real $k$ and any $t > 0$ and any $(x, v)$ in the domain of $G(x, v)$.  For convenience,  we henceforth assume that $k \ne 1$.

In local co-ordinates $(x^a, v^a)$, where $a = 1, 2, \dots, n$, denote by $\partial_a$ and $D_a$, the derivative operators:
\[ \partial_a = \frac{\partial}{\partial x^a}, \hspace{10pt}  D_a = \frac{\partial}{\partial v^a}.\]
These operators mutually commute.  Define the following quantities:
\[ G_a = \partial_a G, \hspace{10pt}g_a = D_a G, \hspace{10pt}g_{ab} = D_aD_b G = g_{ba}, \hspace{10pt}g_{abc} = D_aD_bD_c G = g_{(abc)}.\]
Note that we have:
\[ v^a g_a = v^a D_a G = kG,    \hspace{10pt}v^a g_{ab} = (k - 1)g_b, \hspace{10pt}v^a g_{abc} = (k - 2)g_{bc}.\] 
Then, by the last section, a representative contact one-form for the dynamics is the following one-form $\alpha$, with exterior derivative $\beta$, considered on the space $G = 0$:
\[ \alpha = g_a dx^a, \hspace{10pt} \beta = d\alpha = - (D_{[a} G_{b]})dx^a dx^b + g_{ab} dv^a dx^b.\]
For $V = V^a \partial_a + U^a D_a$ to be a dynamical vector field, we need the conditions:
\[  0 = G,\hspace{10pt}  0 = dG = G_a dx^a + g_a dv^a, \hspace{10pt} 0 = V(G) = V^a G_a + U^a g_a, \]
\[ 0 = V.\alpha  = V^a g_a, \]
\[ 0 = V.\beta - t (k -1)\alpha = - V^a g_{ab} dv^b + (U^a g_{ab} - 2V^aD_{[a} G_{b]} - t(k - 1) g_b)dx^b.\]
Here $t$ is a scalar function.  We infer the relations, for some scalar function $s$:
\[ V^a g_{ab} = s(k - 1)g_b, \]
\[ U^a g_{ab} - 2V^aD_{[a} G_{b]} - t (k -1)g_b = - s(k - 1)G_b.\]
Henceforth we assume that $g_{ab}$ is invertible (the \emph{regular} case).  Then we have the general solutions:
\[ (V^a, U^a)  = t(0,  v^a)  + s(v^a, u^a), \]
\[ u^a g_{ab} =  2v^aD_{[a} G_{b]}  -  (k - 1)G_b =  G_b -   v^a \partial_a  g_b.\]

So a basis for the collection of dynamical vector fields is the pair $\{H, V\}$:
\[ H = v^a D_a, \hspace{10pt} V = v^a \partial_a + u^a D_a, \hspace{10pt} u^a g_{ab} =  G_b -   v^a \partial_a  g_b.\]
Note that we have the required relation:
\[ V(G) = v^a G_a + u^a g_a = v^a G_a + (k -1)^{-1} u^a g_{ab} v^b \]
\[ = (k -1)^{-1}((k -1) v^a G_a + ( G_b -   v^a \partial_a  g_b) v^b) \]
\[ = (k -1)^{-1}(k v^a G_a  -   v^a \partial_a ( v^bg_b)) = 0. \] 
Note that $H$ is the homogeneity operator in $v^a$.  Also $u^a$ is homogeneous of degree $2$ in $v^a$.  Passing down to the sphere bundle $\mathbb{SM}$, only the direction field $V$ survives, giving a foliation of $\mathscr{G}$ by the dynamical curves, the (maximally extended) trajectories of $V$.   We call these dynamical curves null geodesics.  Note that the space of null geodesics, denoted $\mathscr{N}$, which has dimension $2n -3$,  itself carries a natural contact structure induced by that of $\mathscr{G}$.

Summarizing, the hypersurface $\mathscr{G}$ uniquely determines its dynamics, which gives a foliation of $\mathscr{G}$ by a $2n -3$ parameter set of null geodesic curves.  These curves are precisely those that annihilate the contact structure $\alpha_{\mathscr{G}}$ and preserve the contact structure (up to scale).

\subsection{The Ehresmann connection and its characterization}
Introduce the quantity $U_b^{\hspace{3pt}a}$ and the horizontal vector fields $H_a$:
\[ U_b^{\hspace{3pt}a} = - \frac{1}{2} D_b u^a, \hspace{10pt} H_a = \partial_a - U_a^{\hspace{3pt}b}D_b.\]
Note that $U_b^{\hspace{3pt} a} $ is homogeneous of degree one in $v^a$.  Also we have the relations:
\[ v^b U_b^{\hspace{3pt}a} = - u^a, \]
\[ V(g_a) = v^b \partial_b g_a + u^b g_{ab} = G_a, \]
\[ U_b^{\hspace{3pt}a} g_a = \frac{1}{2}(  - D_b (u^a g_a) +   G_b -   v^a \partial_a  g_b)   = \frac{1}{2}(D_b (v^a G_a) +   G_b -   v^a \partial_a  g_b)\]
\[ =  \frac{1}{2}\left(v^a (D_b G_a  - \partial_a g_b) +   2G_b\right) = G_b.\]
This last relation shows that the vector fields $H_a$ are intrinsic to the dynamical surface $G = 0$:
\[ H_a (G) = (\partial_a - U_a^{\hspace{3pt}b}D_b)G = G_a - U_a^{\hspace{3pt}b}g_b = 0.\]
Note that we have the simple formula determining the null geodesic spray $V$ in terms of the horizontal vector fields $H_a$:
\[ V = v^a H_a.\]
The intrinsic tangent vector fields of $G = 0$ are spanned by $\{H_a, L_{bc}\}$, where $L_{bc} = g_{[b} D_{c]}$.
Dually, we introduce the basis of one-forms:
\[ \theta^a = dx^a, \hspace{10pt} \phi^a = dv^a + U_b^{\hspace{3pt}a} dx^b.\]
These are subject to the single linear relation, valid in the space $G = 0$:
\[ g_a \phi^a = g_a dv^a + g_aU_b^{\hspace{3pt}a} dx^b  = g_a dv^a + G_a dx^a = dG = 0.\]
For $\mathscr{G}$ we may use the forms $\theta^a $ and $\phi^{ab} = v^{[a} \phi^{b]}$ as a spanning set.

The horizontal vector fields $H_a$ form an Ehresmann connection over $M$, for the space $\mathscr{G}$.  We now show how to characterize this connection and the corresponding curvature, uniquely.

The curvature generator is (symmetric tensor product here):
\[ \mathcal{G}_0 = \frac{1}{2}g_{ab} dx^a dx^b.\]
We take the Lie derivative of $\mathcal{G}_0$ along $V$, giving the symmetric tensor $\mathcal{G}_1$:
\[  \mathcal{G}_1 = \mathscr{L}_V \mathcal{G}_0 =  g_{ab} dv^a dx^b + \frac{1}{2}V(g_{ab}) dx^a dx^b. \]
Next we Lie derive $\mathcal{G}_1$ along $V$, giving the symmetric tensor $\mathcal{G}_2$:
\[ \mathcal{G}_2 = \mathscr{L}_V \mathcal{G}_1 =  2V(g_{ab}) dv^a dx^b + \frac{1}{2}V^2(g_{ab}) dx^a dx^b + g_{ab} dv^a dv^b + g_{ab} dx^b du^a\]
\[ =  2V(g_{ab}) dv^a dx^b + \frac{1}{2}V^2(g_{ab}) dx^a dx^b + g_{ab} dv^a dv^b + g_{cb} dx^b dv^a D_a u^c + g_{bc} dx^a dx^b \partial_a u^c  \]
\[ =  g_{ab} dv^a dv^b + dv^a dx^b (2V(g_{ab}) + D_a (u^c g_{bc}) - u^c g_{abc})    + \frac{1}{2}dx^a dx^b (V^2(g_{ab}) + 2g_{bc}  \partial_a u^c)  \]
\[ =  g_{ab} dv^a dv^b + dv^a dx^b (2v^c \partial_c g_{ab} + u^c g_{abc}  + D_a (\partial_b G - v^e \partial_{e} g_{b}))    + \frac{1}{2}dx^a dx^b (V^2(g_{ab}) + 2g_{bc}  \partial_a u^c)  \]
\[ =  g_{ab} dv^a dv^b + dv^a dx^b (2v^c \partial_c g_{ab} + u^c g_{abc}  + \partial_b g_a -  \partial_a g_{b} -  v^e \partial_{e} g_{ab})    + \frac{1}{2}dx^a dx^b (V^2(g_{ab}) + 2g_{bc}  \partial_a u^c)  \]
\[ =  g_{ab} dv^a dv^b + dv^a dx^b (V( g_{ab})   + \partial_b g_a -  \partial_a g_{b})    + \frac{1}{2}dx^a dx^b (V^2(g_{ab}) + 2 \partial_a (u^cg_{bc}) - 2 u^c \partial_a g_{bc})  \]
\[ =  g_{ab} dv^a dv^b + dv^a dx^b (V( g_{ab})   + \partial_b g_a -  \partial_a g_{b})    + \frac{1}{2}dx^a dx^b (V^2(g_{ab}) + 2 \partial_a \partial_b G - 2v^c \partial_{c} \partial_a g_{b}  - 2 u^c \partial_a g_{bc})  \]
\[ =  g_{ab} dv^a dv^b + dv^a dx^b (V( g_{ab})   + \partial_b g_a -  \partial_a g_{b})    + \frac{1}{2}dx^a dx^b (V^2(g_{ab})  - 2V(\partial_a g_{b}) + 2 \partial_a \partial_b G).  \]

Now recall that:
\[ u^a g_{ab} = \partial_b G - v^e \partial_{e} g_{b}, \hspace{10pt} U_a^{\hspace{3pt} b} = - 2^{-1} D_a u^b, \]
\[ v^a U_a^{\hspace{3pt} b} = - u^b, \hspace{10pt} U_a^{\hspace{3pt}b} g_b =  G_a, \]
\[ \phi^a = dv^a + U_{b}^{\hspace{3pt}a} dx^b, \hspace{10pt} \theta^a = dx^a.\] 
Then we have:
\[ U_{ab} =  U_a^{\hspace{3pt} c}g_{bc} = - 2^{-1} D_a (u^b g_{bc}) + 2^{-1} u^b g_{abc}\]
\[ =  - 2^{-1} D_a ( \partial_b G - v^e \partial_{e} g_{b}) + 2^{-1} u^b g_{abc}\]
\[ =  - 2^{-1}\partial_b g_a  +2^{-1}  \partial_{a} g_{b} + 2^{-1} v^e \partial_e g_{ab} + 2^{-1} u^b g_{abc}\]
\[ = 2^{-1}(V(g_{ab}) + \partial_a g_b - \partial_b g_a), \]
\[ v^a U_{ab} = - u_b, \hspace{10pt} U_{ab} v^b = G_a.\]
Substituting into our expression for $\mathcal{G}_2$ and completing the square on the terms involving $dv^a$, we have:
\[ \mathcal{G}_2 =  g_{ab} dv^a dv^b + dv^a dx^b (V( g_{ab})   + \partial_b g_a -  \partial_a g_{b})    + \frac{1}{2}dx^a dx^b (V^2(g_{ab})  - 2V(\partial_a g_{b}) + 2 \partial_a \partial_b G)\]
\[ =  g_{ab} dv^a dv^b + 2dv^a dx^b U_{ba}    + \frac{1}{2}dx^a dx^b (V^2(g_{ab})  - 2V(\partial_a g_{b}) + 2 \partial_a \partial_b G)\]
\[ =   g_{ab} (dv^a + U_{c}^{\hspace{3pt}a} dx^c)(dv^b +  dx^d U_{d}^{\hspace{3pt} b})    + \frac{1}{2}dx^a dx^b (V^2(g_{ab})  - 2V(\partial_a g_{b}) + 2 \partial_a \partial_b G) - g_{ab}U_{c}^{\hspace{3pt}a} U_{d}^{\hspace{3pt} b}dx^c dx^d\]
\[ =    g_{ab} \phi^a \phi^b  + \frac{1}{2}dx^a dx^b (V^2(g_{ab})  - 2V(\partial_a g_{b}) + 2 \partial_a \partial_b G - 2g_{cd}U_{a}^{\hspace{3pt}c} U_{b}^{\hspace{3pt} d}).\]
So we have now:
\[ \mathcal{G}_2 =  g_{ab} \phi^a \phi^b  + \theta^a \theta^b  S_{ab}, \]
\[ S_{ab} = 2^{-1}\left( V^2(g_{ab})  - 2V(\partial_a g_{b}) + 2 \partial_a \partial_b G - 2g_{cd}U_{a}^{\hspace{3pt}c} U_{b}^{\hspace{3pt} d}\right) = S_{ba}.\]
The tensor $S_{ab}$ generalizes the sectional curvature tensor of Lorentzian geometry.
Back substituting into $\mathcal{G}_1$, we have, since the skew terms in $U_{ba}$ cancel with the symmetric $dx^a dx^b$:
\[ \mathcal{G}_1 =  g_{ab} \theta^a \phi^b.\]
Now suppose we modify the Ehresmann connection $\phi^a$ to $\psi^a = \phi^a  + \theta^b u_{bc}g^{ac}$, for some $u_{ba}$.  Then we have $ \mathcal{G}_1 =  g_{ab} \psi^a \theta^b$ provided that $u_{ab}= u_{ba}$.  Also expressed in terms of $\psi^a$ and $\theta^a$, we have:
\[ \mathcal{G}_2 =  g_{ab} \psi^a \psi^b - 2\psi^a \theta^b u_{ba} +  \theta^a \theta^b  u_{bc} u_{bd} g^{cd}       +  \theta^a \theta^b  S_{ab}. \]
We see that the cross term in $\psi^a \theta^b$ vanishes if and only $u_{ab} = 0$.  We have proved that the Ehresmann connection is uniquely characterized by the absence of $\theta^a \phi^b $ terms in $\mathcal{G}_2$, so by the decomposition:
\[  \mathcal{G}_2 =  g_{ab} \phi^a \phi^b  + \theta^a \theta^b  S_{ab}.\]
Then, as a bonus, we have the formula: 
\[  \mathcal{G}_1 =  g_{ab} \theta^a \phi^b.\]

\subsection{The curvature and its relation to the sectional curvature}
For the curvature of the Ehresmann connection, we have:
\[ [H_a, H_b] =  [\partial_a - U_a^{\hspace{3pt} c} D_c,  \partial_b - U_b^{\hspace{3pt} d} D_d]\]
\[ = - 2(\partial_{[a} U_{b]}^{\hspace{4pt}c}  - U_{[a}^{\hspace{3pt} d} D_{|d|} U_{b]}^{\hspace{3pt}c})D_c  = - 2R_{ab}^{\hspace{8pt}c} D_c, \]
\[ R_{ab}^{\hspace{8pt}c} = \partial_{[a} U_{b]}^{\hspace{4pt}c}  - U_{[a}^{\hspace{4pt} d} D_{|d|} U_{b]}^{\hspace{3pt}c}  = H_{[a} U_{b]}^{\hspace{4pt}c}.\] 
Note the relation:
\[ R_{ab}^{\hspace{8pt}c} g_c = \partial_{[a}( U_{b]}^{\hspace{5pt}c}g_c) + U_{[a}^{\hspace{5pt}c}\partial_{b]} g_c - U_{[a}^{\hspace{5pt}d} D_{|d|} (U_{b]}^{\hspace{5pt}c} g_c)\]
\[  = \partial_{[a} G_{b]}  + U_{[a}^{\hspace{5pt}c}(\partial_{b]}g_c -  D_{|c|}G_{b]}) = 0.\]
We extend the vector fields $H_a$ and $D_a$ to act on forms, as derivations of degree zero,  by requiring that they annihilate the forms $\theta^a$ and $\phi^a$.  Introduce the derivations of degree minus one, denoted $\delta_a$ and $\epsilon_a$,  dual to $\theta^a$ and $\phi^a$, which obey, in particular, the relations:
\[ \delta_a \theta^b = \epsilon_a \phi^b = \delta_a^b, \hspace{10pt}  \delta_a \phi^b = \epsilon_a \theta^b = 0. \]  
The operators $\delta_a$ and $\epsilon_a$ mutually  anti-commute and commute with $H_a$ and $D_a$.  Since $d\theta^a = 0$,  we have the expression for the exterior derivative:
\[ d   = \theta^a H_a + \phi^a D_a  + (d \phi^a) \epsilon_a. \]
Define the covariant exterior derivative:
\[ \partial = dx^a H_a.\]
Note the relation:
\[ \partial v^a = dx^b H_b v^a = - U^a, \hspace{10pt} U^a = dx^b U_b^{\hspace{3pt} a}.\]
Then, using forms, we have the curvature two-form operator $R$ and the curvature two-form  $R^c$:
\[ R = - \partial^2 =  - (dx^a H_a)^2  = (\partial U^c)D_c =  R^c D_c, \] 
\[ R^c =   \partial U^c = dx^a dx^b R_{ab}^{\hspace{8pt}c}  = dx^a dx^b (\partial_{a} U_{b}^{\hspace{4pt}c} - U_{a}^{\hspace{3pt} d} D_{d} U_{b}^{\hspace{3pt}c}) = dx^a dx^b H_{[a} U_{b]}^{\hspace{6pt} c}.\]
The Bianchi identity for the covariant derivative, $\partial$, is:
\[ \partial R^b = \partial^2 U^b =  - R^a D_a U^b.\]
Written out this is:
\[ dx^a dx^b dx^c (H_a R_{bc}^{\hspace{8pt} d}) = - dx^a dx^b dx^c R_{ab}^{\hspace{8pt} e} D_e U_{c}^{\hspace{3pt}d}, \]
\[ H_{[a}R_{bc]}^{\hspace{10pt} d} + R_{[ab}^{\hspace{10pt} e} D_{c]}U_{e}^{\hspace{3pt}d} = 0.\]

Introduce also the Lorentz generators, which are tangent to $\mathscr{G}$ and which, together with $H_a$ span the tangent space to $\mathscr{G}$:
\[ L_{ab} = \frac{1}{2}(g_a D_b - g_b D_a) = g_{[a} D_{b]}.\]
We have the commutators:
\[ [L_{ab}, L_{cd}] = [ g_{[a} D_{b]} , g_{[c} D_{d]}] =   g_{[a} g_{b][c} D_{d]} -  g_{[c} g_{d][a} D_{b]}  =  g_{[c[b} g_{a]} D_{d]} -  g_{[a[d}g_{c]} D_{b]}  \]
\[ = - 2 g_{[c[a} L_{b]d]}  =  2 g_{[a[c} L_{d]b]}.\] 
\[ [H_{a}, L_{bc}] = \frac{1}{2}  [H_a, g_b D_c - g_c D_b] = (H_a g_{[b}) D_{c]} + g_{[b}(D_{c]} U_a^{\hspace{3pt}e}) D_e\]
\[ = (H_a g_{[b}) D_{c]} - (D_{[b} U_{|a|}^{\hspace{6pt}e}) g_{c]} D_e\]
\[ = (H_a g_{[b}) D_{c]} - 2(D_{[b} U_{|a|}^{\hspace{6pt}e}) L_{c]e}  -  g_e(D_{[b} U_{|a|}^{\hspace{6pt}e}) D_{c]}   \]
\[ = (H_a g_{[b}) D_{c]} - 2(D_{[b} U_{|a|}^{\hspace{6pt}e}) L_{c]e}  -  (D_{[b} (g_e U_{|a|}^{\hspace{6pt}e})) D_{c]}  +  U_{a[b} D_{c]}  \]
\[ = (H_a g_{[b}) D_{c]} - 2(D_{[b} U_{|a|}^{\hspace{6pt}e}) L_{c]e}  -  (D_{[b} G_{|a|}) D_{c]}  +  U_{a[b} D_{c]}  \]
\[ = (H_a g_{[b}) D_{c]} - 2(D_{[b} U_{|a|}^{\hspace{6pt}e}) L_{c]e}  - ( \partial_a g_{[b}) D_{c]}  +  U_{a[b} D_{c]}  \]
\[ = - U_a^e  g_{e[b} D_{c]} - 2(D_{[b} U_{|a|}^{\hspace{6pt}e}) L_{c]e}   +  U_{a[b} D_{c]} =   - (D_{[b} U_{|a|}^{\hspace{6pt}e}) L_{c]e}. \]
So we have:
\[  [H_{a}, L_{bc}] = - 2(D_{[b} U_{|a|}^{\hspace{6pt}e}) L_{c]e}   = - 2(D_a U_{[b}^{\hspace{6pt}e}) L_{c]e}   = U_{a[b}^{\hspace{9pt} e} L_{c]e},\]
\[  U_{ab}^{\hspace{6pt} c} = -2 D_a U_{b}^{\hspace{3pt}c} =   D_a D_b u^c = U_{ba}^{\hspace{6pt} c}.\]
Our second version of the sectional curvature is defined as follows:
\[      T^b =   v^a \delta_a R^b = 2v^c dx^a R_{ca}^{\hspace{8pt}b}  = v^a \delta_a \partial U^b\]
\[=  v^a [\delta_a,   \partial ] U^b - v^a  \partial U_a^{\hspace{3pt}b}  =  v^a H_a U^b - v^a  \partial  U_a^{\hspace{3pt}b}\] 
\[=  v^a H_a U^b -   \partial  (v^a U_a^{\hspace{3pt} b}) +  U_{a}^{\hspace{3pt} b} \partial  v^a \] 
\[=  VU^b +   \partial u^b - U^a  U_{a}^{\hspace{3pt} b}   = dx^a T_a^{\hspace{3pt} b}, \]
\[ T_a^{\hspace{3pt} b} = 2v^c R_{ca}^{\hspace{8pt}b}  =   V(U_a^{\hspace{3pt}b}) +  H_a u^b -  U_a^{\hspace{3pt}c} U_{c}^{\hspace{3pt} b}.\]

Lowering the upper index of $T_{a}^{\hspace{3pt}b}$, we have:
\[ T_{ab} = T_a^{\hspace{3pt} c} g_{bc} = V(U_a^{\hspace{3pt} c})g_{bc} + g_{bc} H_a u^c -  U_a^{\hspace{3pt}c} U_{cb}\]
\[ = V(U_{ab}) - U_a^{\hspace{3pt} c}V(g_{bc}) + H_a (u^cg_{cb}) - u^c  H_a g_{bc}  -  U_a^{\hspace{3pt}c} U_{cb}.\]
So now we compute:
\[ 2T_{ab} - 2S_{ab} =  2T_{ab} -  V^2(g_{ab}) + 2V(\partial_{(a} g_{b)})  - 2\partial_a \partial_b G + 2g_{cd} U_a^{\hspace{3pt}c}  U_b^{\hspace{3pt}d} \]
\[   = 2V(U_{ab}) - 2U_a^{\hspace{3pt} c}V(g_{cb}) + 2H_a (\partial_b G - v^c \partial_c g_b) - 2u^c  H_a g_{cb} \]\[ -  2U_a^{\hspace{3pt}c} (U_{cb} - U_{bc}) - 2\partial_a \partial_b G - V^2 (g_{ab}) + 2 V \partial_{(a} g_{b)}\]
\[   = V^2(g_{ab}) + V(\partial_a g_b - \partial_b g_a) - 2U_a^{\hspace{3pt} c}V(g_{cb}) + 2H_a (\partial_b G - v^c \partial_c g_b) - 2u^c H_a g_{cb} \]\[ -  2U_a^{\hspace{3pt}c} (U_{cb} - U_{bc}) - 2\partial_a \partial_b G - V^2 (g_{ab}) + V( \partial_{a} g_{b} + \partial_b g_a) = 2Y_{ab}, \]
\[   Y_{ab} =   V(\partial_a g_b) - U_a^{\hspace{3pt} c}V(g_{cb}) + (\partial_a  - U_a^{\hspace{3pt}e} D_e) (\partial_b G - v^c \partial_c g_b) - u^c  (\partial_a  - U_a^{\hspace{3pt}e} D_e) g_{cb}  -  U_a^{\hspace{3pt}c} (\partial_{c} g_b - \partial_b g_c) - \partial_a \partial_b G\]
\[  =   V(\partial_a g_b) - U_a^{\hspace{3pt} c}V(g_{cb})  - v^c\partial_c \partial_a g_b   - U_a^{\hspace{3pt}e} D_e (\partial_b G - v^c \partial_c g_b) - u^c  (\partial_a  - U_a^{\hspace{3pt}e} D_e) g_{cb}  -  U_a^{\hspace{3pt}c} (\partial_{c} g_b - \partial_b g_c) \]
\[ = u^eD_e \partial_a g_b - U_a^{\hspace{3pt} c}V(g_{cb})    - U_a^{\hspace{3pt}e} \partial_b g_e + U_a^{\hspace{3pt}e} D_e ( v^c \partial_c g_b) - u^c  (\partial_a  - U_a^{\hspace{3pt}e} D_e) g_{cb}  -  U_a^{\hspace{3pt}c} (\partial_{c} g_b - \partial_b g_c) \]
\[ = u^e\partial_a g_{eb} - U_a^{\hspace{3pt} c}V(g_{cb})    - U_a^{\hspace{3pt}e} \partial_b g_e + U_a^{\hspace{3pt}e} \partial_e g_b  + U_a^{\hspace{3pt}e}v^c   \partial_c g_{eb} - u^c  (\partial_a  - U_a^{\hspace{3pt}e} D_e) g_{cb}  -  U_a^{\hspace{3pt}c} (\partial_{c} g_b - \partial_b g_c) \]
\[ = u^e\partial_a g_{eb} - U_a^{\hspace{3pt} c}V(g_{cb})  + U_a^{\hspace{3pt}e}v^c   \partial_c g_{eb} - u^c  (\partial_a  - U_a^{\hspace{3pt}e} D_e) g_{cb}   \]
\[ = u^e\partial_a g_{eb} - U_a^{\hspace{3pt} c}u^e D_e g_{cb}   - u^c  (\partial_a  - U_a^{\hspace{3pt}e} D_e) g_{cb}   \]
\[ = u^e\partial_a g_{eb}     - u^c  \partial_a g_{cb}   = 0.\]
In going from the penultimate to the last line we used that the tensor $D_e g_{cb} = g_{ecb} = D_e D_cD_b G$ is totally symmetric.
So we have proved that the sectional curvature is simply the contraction with $v^a$ with the curvature, up to a constant factor:
\[ v^a \delta_a R^b =   dx^a S_{a}^{\hspace{3.5pt}b}, \hspace{10pt}S_a^{\hspace{3.5pt}b} = S_{ac}g^{bc}, \]
\[ 2v^a R_{ab}^{\hspace{8pt}c} = S_b^{\hspace{3pt} c} = S_{ab} g^{ac}.\]
Note the relation implied by this, using the fact that $R_{ab}^{\hspace{8pt}c}$ is skew in $a$ and $b$:
\[ v^a S_{ab} = 0.\]

We check this relation directly:
\[ v^a S_{ab} = v^a (V^2(g_{ab})  - 2V(\partial_{(a} g_{b)}) + 2 \partial_a \partial_b G  -   2g_{cd} U^{\hspace{3pt}c}_a U^{\hspace{3pt}d}_b) \]
\[ = v^a V^2(g_{ab})  - 2v^a V(\partial_{(a} g_{b)}) + 2v^a \partial_a \partial_b G  +  2g_{cd} u^c U^{\hspace{3pt}d}_b\]
\[ = [v^a,  u^e D_e]V(g_{ab}) + V (v^a V g_{ab})- 2[v^a, u^e D_e](\partial_{(a} g_{b)}) - 2 V (v^a \partial_{(a} g_{b)})  + 2v^a \partial_a \partial_b G  +  2g_{cd} u^c U^{\hspace{3pt}d}_b\]
\[ =  - u^aV(g_{ab}) + V ([v^a,  V] g_{ab}) + (k -1)V^2 g_b  + 2u^a\partial_{(a} g_{b)} -  V (v^a \partial_{a} g_{b}) - kV \partial_{b} G   + 2v^a \partial_a \partial_b G  +  2g_{cd} u^c U^{\hspace{3pt}d}_b\]
\[ =  - 2u^aV(g_{ab}) - g_{ab}V (u^a )  - V \partial_b  G + 2u^a\partial_{(a} g_{b)} -  u^a \partial_{a} g_{b} -  v^a V( \partial_{a} g_{b})  + 2v^a \partial_a \partial_b G  +  2g_{cd} u^c U^{\hspace{3pt}d}_b\]
\[ =  - 2u^aV(g_{ab}) - g_{ab}V(u^a)  - V \partial_b  G + u^a\partial_{b} g_{a} -  v^a V( \partial_{a} g_{b})  + 2v^a \partial_a \partial_b G  +  u^a (V(g_{ab}) + \partial_b g_a - \partial_a g_b)\]
\[ =  - V (u^a  g_{ab})  -  v^a V( \partial_{a} g_{b})  + v^a \partial_a \partial_b G  +  u^a (\partial_b g_a - \partial_a g_b)\]
\[= - V(\partial_b G - v^c \partial_c g_b) - v^a V( \partial_{a} g_{b})  + v^a \partial_a \partial_b G  +  u^a (\partial_b g_a - \partial_a g_b)\]
\[= - u^a \partial_b g_a  + (V( v^a)) \partial_a g_b  +  u^a (\partial_b g_a - \partial_a g_b) = 0.\]
Finally we wish to show that the sectional curvature $S_{b}^{\hspace{3pt}c} = g^{ac}S_{ab}$ determines the full curvature tensor.
We start with the formula, just proved above:
\[  S_{b}^{\hspace{3pt}c} = 2v^a R_{ab}^{\hspace{8pt}c}, \]
Take the curl of both sides with $D_a$ giving:
\[ D_{[a} S_{b]}^{\hspace{5pt}c} = 2R_{ab}^{\hspace{8pt}c} - 2v^e D_{[a} R_{b]e}^{\hspace{10pt}c}.\]
We need to analyze the last term, so we first recall the formula for $2R_{be}^{\hspace{8pt}c}$:
\[ 2R_{be}^{\hspace{8pt}c} = \partial_b  U_{e}^{\hspace{3pt} c}  - \partial_e  U_{b}^{\hspace{3pt}c} -  U_b^{\hspace{3pt} d} D_d U_e^{\hspace{3pt} c} +   U_e^{\hspace{3pt} d} D_d U_b^{\hspace{3pt} c}. \]
Then we have, using repeatedly the relation $D_a U_b^{\hspace{3pt}c} = D_b U_a^{\hspace{3pt} c}$ and the fact that $U_a^{\hspace{3pt} b}$ is homogeneous of degree one in $v^a$, so $v^c D_c U_a^{\hspace{3pt} b} = U_a^{\hspace{3pt} b}$:
\[ 2v^e D_{[a}R_{b]e}^{\hspace{10pt}c}  = v^e D_{[a} \partial_{b]} U_{e}^{\hspace{3pt}c} + v^e U_{[a}^{\hspace{5pt}g} D_{b]} D_g U_e^{\hspace{3pt}c} + v^e (D_{[a} U_{|e|}^{\hspace{7pt}g} ) D_{|g|} U_{b]}^{\hspace{5pt}c} \]
\[ = - v^e D_{e} \partial_{[a} U_{b]}^{\hspace{5pt}c} +  U_{[a}^{\hspace{5pt}g} v^e D_{|e|} D_{b]}  U_g^{\hspace{3pt}c}  + (v^e D_e U_{[a}^{\hspace{5pt}g}) D_{|g|} U_{b]}^{\hspace{3pt}c} \]
\[ = -  \partial_{[a} U_{b]}^{\hspace{5pt}c}  +  U_{[a}^{\hspace{5pt}g}   D_{|g|} U_{b]}^{\hspace{5pt}c}   = - R_{ab}^{\hspace{8pt} c}.\]
So we have the required relation:
\[ D_{[a} S_{b]}^{\hspace{5pt}c} = 2R_{ab}^{\hspace{8pt}c} - 2v^e D_{[a} R_{b]e}^{\hspace{10pt}c}  =  2R_{ab}^{\hspace{8pt}c} - (-  R_{ab}^{\hspace{8pt} c})  = 3R_{ab}^{\hspace{8pt}c}, \hspace{10pt} D_{[a}R_{bc]}^{\hspace{10pt} d}  = 0.\]
So we have proved that the sectional curvature determines the full curvature and conversely! 
Consequently we have the relations, using the Lorentz generators, instead of $D_a$:
\[ L_{[ab} S_{c]}^{\hspace{5pt}d} =  g_{[a} D_b S_{c]}^{\hspace{5pt}d} = 6g_{[a} R_{bc]}^{\hspace{10pt}d}, \hspace{10pt} L_{[ab} R_{cd]}^{\hspace{10pt} e} = 0.\]

\subsection{Generalized conformal transformations}
A generalized conformal transformation is the transformation from the Lagrangian $G(x, v)$ to the Lagrangian $H(x, v)$, where we have:
\[  G(x, v) = H(x, v)J(x, v)^{-1}.\]
Here $J(x, v)$ is a non-zero function, smooth and defined in a neighbourhood of the cone $G(x, v) =0$.  We have the homogeneities, valid for any real $t> 0$:
\[ G(x, tv) = t^k G(x, v), \hspace{10pt} H(x, tv) = t^p H(x, v), \hspace{10pt} J(x, tv) = t^q J(x, v), \]
\[  p - q = k \ne 1,  \hspace{10pt} p \ne 1.\]
Write $J = e^s$, where $v^a D_a s = q$.  Taking derivatives, we have:
\[ H = e^s G, \]
\[ h_i =e^s( g_i + s_i G), \]
\[ j_i = e^s s_i, \hspace{10pt} j_{ij} = e^s(s_{ij} + s_i s_j), \]
\[ h_{ij} = e^s ((s_{ij} + s_is_j) G + s_i g_j + s_j g_i +  g_{ij}). \]
Here we have:
\[ [g_i, j_i, h_i, s_i] = D_i[g, h, j, s], \hspace{10pt} [g_{ij}, j_{ij}, h_{ij}, s_{ij}] = D_iD_j[g, h, j, s] =   [g_{ji}, j_{ji}, h_{ji}, s_{ji}].\]
We also put:
\[ [G_i, J_i, H_i, S_i] = \partial_i[G, H, J, S], \hspace{10pt}  J_i = e^s S_i.\]
Note that we have:
\[ H_j = J_j G + J G_j = e^s(G_j + S_j G).\]

We assume henceforth that the matrix $g_{ij}$ is invertible.  Using the relations  $v^i g_{ij} = (k - 1)g_j$, $v^j g_j = kG$ and $v^i v^j g_{ij} = (k - 1)v^j g_j = k(k -1) G$, we have:
\[e^{-s} h_{ij} =  (s_{ij} + s_is_j) G + s_i g_j + s_j g_i +  g_{ij}\]
\[ =   g_{pq}( \delta^{p}_{i} + (k - 1)^{-1}v^p s_i)( \delta^{q}_{j} +  (k - 1)^{-1} v^q s_j) + G( s_{ij}-(k -1)^{-1} s_is_j).\]
On the cone $G =0$, we deduce the relations:
\[ e^{-s} h_{ij} =  g_{pq}M^p_i M^q_j, \]
\[ M^j_i =  \delta^{j}_{i} +  (k - 1)^{-1}v^j s_i, \hspace{10pt}(M^{-1})^i_j =  \delta^{i}_{j} -   (p - 1)^{-1}v^i s_j, \]
On the cone $G = 0$, the matrix $e^{-s}h_{ij}$ is invertible, with inverse $e^s h^{ij}$ given by:
\[ e^s h^{ij} =  g^{pq} (M^{-1})_p^i(M^{-1})_q^j.\]
Then, on $G =0$, the signatures of $g_{ij}$ and $h_{ij}$ are equal if $J > 0$ and opposite, if $J < 0$.  Note that it follows that $h_{ij}$ is invertible in a neighbourhood of $G =0$, with inverse still denoted $h^{ij}$.
Contracting the equation for $e^{-s} h_{ij}$ with $u^i$, using the fact that $ u^a g_a + v^a G_a   = V(G) = 0$, we get:
\[  h_{ij} u^i = e^s u^i(s_{ij} + s_is_j) G + u^i s_i e^s g_j + s_j u^i e^s g_i + e^s \partial_j (e^{-s}H) - e^s v^m \partial_m (e^{-s}h_j - s_j G) \]
\[ \hspace{-40pt} = H_j - v^m \partial_m h_j +    e^s u^i(s_{ij} + s_is_j) G + u^i s_i e^s g_j + s_j u^i e^s g_i -  e^s G S_j + e^s v^m \partial_m ( s_j G) + v^m S_m e^s  (g_j + s_j G)  \]
\[ = H_j - v^m \partial_m h_j + V(s) e^s g_j +   e^sG( u^is_{ij} + u^is_is_j-   S_j +  v^m \partial_m s_j     + v^m S_m   s_j)  \]
\[ = H_j - v^m \partial_m h_j + V(s) e^s g_j +   e^sG(-   S_j +  V( s_j)     + V(s)   s_j) \]
\[ = H_j - v^m \partial_m h_j + V(s) h_j  +   e^sG(-   S_j +  V( s_j) ).  \]
So we have:
\[ 
u^i = h^{ij}( H_j  - v^k \partial_k  h_j)  + V(s) (p - 1)^{-1} v^i  +e^s Gh^{ij}(V( s_j) - S_j).\]
We solve this equation for $u^a$, order by order in $G$.  To order $G$ we see that:
\[ u^a = \gamma^a +  \delta v^a + G \epsilon^a, \]
\[ \gamma^a = h^{ab}(H_b  - v^k \partial_k  h_b), \hspace{10pt} \delta =  W(s)(k -1)^{-1}, \]
\[ V = W + \delta v^a D_a + G \epsilon^a D_a, \]
\[ W = v^a \partial_a + \gamma^a D_a. \]
Substituting, we need the relation:
\[ \hspace{-45pt} 0 = h^{ij}( H_j  - v^k \partial_k  h_j)  + V(s) (p - 1)^{-1} v^i  + e^sGh^{ij}(V( s_j) - S_j) - h^{ij}(H_j  - v^k \partial_k  h_bj) - v^iW(s)(k -1)^{-1} - G \epsilon^i\]
\[  =  (W(s) + \delta q + G \epsilon^a s_a) (p - 1)^{-1} v^i  + G(- \epsilon^i + e^s h^{ij}(V( s_j) - S_j) - v^iW(s)(k -1)^{-1}.\]
The terms not proportional to $G$ cancel,  so factoring out $G$, we get the relation:
\[ \epsilon^i =  \epsilon^a s_a (p - 1)^{-1} v^i + e^s h^{ij}(V( s_j) - S_j),  \]
\[ \epsilon^i =  \epsilon^a s_a (p - 1)^{-1} v^i + e^s h^{ij}(W( s_j) - \delta  s_j  + G \epsilon^a s_{aj} - S_j) \]
Next we have:
\[ \epsilon^a = \zeta^a +  \eta v^a + G \theta^a, \]
\[ \zeta^a =  e^s h^{ij}(W( s_j) - \delta  s_j - S_j), \]
\[ \eta = (k -1)^{-1} \zeta^a s_a.\]
\[ 0 =   (\zeta^a +  \eta v^a + G \theta^a) s_a (p - 1)^{-1} v^i + e^s h^{ij}(W( s_j) - \delta  s_j  + G \epsilon^a s_{aj} - S_j) -  \zeta^i -  \eta v^i -  G \theta^i\]
\[ =    ((p - 1)\eta+ G \theta^a s_a) (p - 1)^{-1} v^i + Ge^s h^{ij} \epsilon^a s_{aj} -  \eta v^i -  G \theta^i, \]
\[ \theta^i  =   \theta^a s_a (p - 1)^{-1} v^i + e^s h^{ij} \epsilon^a s_{aj}, \]
\[ \theta^i  =   \theta^a s_a (p - 1)^{-1} v^i + Ge^s h^{ij} \theta^a s_{aj} + e^s h^{ij} (\zeta^a +  \eta v^a) s_{aj}.\]
Note that the residual term $\theta^i$ obeys the linear matrix equation:
\[ \theta^i  -   \theta^a s_a (p - 1)^{-1} v^i - Ge^s h^{ij} \theta^a s_{aj} =  e^s h^{ij} (\zeta^a +  \eta v^a) s_{aj}, \]
\[ (\delta_a^b - s_a (p - 1)^{-1} v^b + Ge^s h^{bc}  s_{ac})\theta^b =   e^s h^{ij} (\zeta^a +  \eta v^a) s_{aj}, \]
\[ P(\theta) = Q, \]
\[ P_a^{\hspace{3pt} b} = \delta_a^b - s_a (p - 1)^{-1} v^b + Ge^s h^{bc}  s_{ac}, \]
\[ Q^a =  e^s h^{ij} (\zeta^a +  \eta v^a) s_{aj}.\]
When $G =0$, $P_a^{\hspace{3pt} b}$ has inverse the matrix $M_a^{\hspace{3pt} b}$ given above, so the matrix $P$ is always invertible in a neighbourhood of the surface $G = 0$ and then $\theta$ is given by $\theta = P^{-1} Q$.  In particular the required solution exists and is unique in a neighbourhood of $G = 0$, as required.  Also on the surface $G =0$, $\theta$ is given explicitly by the equation $\theta = M(Q)$ (i.e. $\theta^b = M_a^{\hspace{3pt} b}Q^a$).

Summarizing, we have now the expansion:
\[ u^a = \gamma^a + \delta v^a + G\zeta^a + G \eta v^a + G^2 \theta^a\]
Then on $G =0$, we get:
\[ u^a = \gamma^a + \delta v^a, \]
\[ U_b^{\hspace{3pt}a} = - 2^{-1}(D_b \gamma^a  + v^a D_b \delta + \delta \delta_b^a + g_b \zeta^a + v^a g_b \eta).\]
Here we have:
\[ \gamma^a = h^{ab}(H_b  - v^k \partial_k  h_b), \]
\[  \delta =  W(s)(k -1)^{-1}, \]
\[ \zeta^a =  e^s h^{ij}(W( s_j) - \delta  s_j - S_j), \]
\[ \eta = (k -1)^{-1} \zeta^a s_a.\]
We apply these formulas to determine the conformal transformation of the sectional curvature.  We have, evaluating on $G = 0$:
\[ T_{a}^{\hspace{3pt}b} = V(U_{a}^{\hspace{3pt}b})   + H_a(u^b) - U_a^{\hspace{3pt}c} U_c^{\hspace{3pt}b}.\]
Using primes to denote the corresponding quantities computed for the conformally rescaled function $H$, we have:
\[  (T')_{a}^{\hspace{3pt}b} = V'((U')_{a}^{\hspace{3pt}b})   + H'_a((u')^b) - (U')_a^{\hspace{3pt}c} (U')_c^{\hspace{3pt}b}.\]
Now, on $G = 0$, we have:
\[ V = v^a\partial_a + u^a D_a =   W + \delta  v^aD_a, \]
\[ V' = W = v^a \partial_a + \gamma^a D_a, \]
\[ (u' - u)^a = - \delta v^a, \] 
\[ H'_a - H_a = - (U'- U)_a^{\hspace{3pt} b} D_b, \]
\[ (U' - U)_a^{\hspace{3pt}b} =  \frac{1}{2}( v^b D_a (\delta)+  g_a(\zeta^b +  \eta v^b )+ \delta \delta_a^b).\]

Using these relations, we have:
\[ \hspace{-10pt} (T')_{a}^{\hspace{3pt}b}  - T_a^{\hspace{3pt}b} = W((U')_{a}^{\hspace{3pt}b})  - (W + \delta v^c D_c) U_{a}^{\hspace{3pt} b}  + H'_a((u')^b)  - H_a u^b - (U')_a^{\hspace{3pt}c} (U')_c^{\hspace{3pt}b} +  U_a^{\hspace{3pt}c} U_c^{\hspace{3pt}b}\]
\[ = W((U'- U)_{a}^{\hspace{3pt}b})  - \delta U_{a}^{\hspace{3pt} b}  + H'_a((u'- u)^b) + (H'_a - H_a)u^b - (U')_a^{\hspace{3pt}c} (U')_c^{\hspace{3pt}b} +  U_a^{\hspace{3pt}c} U_c^{\hspace{3pt}b}\]
\[ = W((U'- U)_{a}^{\hspace{3pt}b})  - \delta U_{a}^{\hspace{3pt} b}  + H'_a((u'- u)^b)  - (U'- U)_a^{\hspace{3pt} c} D_c u^b - (U')_a^{\hspace{3pt}c} (U')_c^{\hspace{3pt}b} +  U_a^{\hspace{3pt}c} U_c^{\hspace{3pt}b}\]
\[\hspace{-5pt} = W((U'- U)_{a}^{\hspace{3pt}b})  - \delta U_{a}^{\hspace{3pt} b}  - H'_a(\delta v^b)  +2 (U'- U)_a^{\hspace{3pt} c} U_c^{\hspace{3pt}b} - (U')_a^{\hspace{3pt}c} (U')_c^{\hspace{3pt}b} +  U_a^{\hspace{3pt}c} U_c^{\hspace{3pt}b}\]
\[ \hspace{-18pt} = W((U'- U)_{a}^{\hspace{3pt}b})  - \delta U_{a}^{\hspace{3pt} b}  - H'_a(\delta v^b)  - (U'- U)_a^{\hspace{3pt}c} (U'- U)_c^{\hspace{3pt}b}  - [U' - U, U]_a^b\]
\[ = W((U'- U)_{a}^{\hspace{3pt}b})  + \delta (U'- U)_{a}^{\hspace{3pt} b}  - v^b H'_a(\delta) - (U'- U)_a^{\hspace{3pt}c} (U'- U)_c^{\hspace{3pt}b}  - [U' - U, U]_a^b.\]
We write this expression out,  working modulo multiples of $v^b$ and $g_a$ and working on $G = 0$:
\[ 4(T' - T)_a^{\hspace{3pt}b} = 4W((U'- U)_{a}^{\hspace{3pt}b})  + 4\delta (U'- U)_{a}^{\hspace{3pt} b}  - 4(U'- U)_a^{\hspace{3pt}c} (U'- U)_c^{\hspace{3pt}b}  - 4[U' - U, U]_a^b\]
\[\hspace{-45pt} =  2W(v^b D_a (\delta)+  g_a(\zeta^b +  \eta v^b ) + \delta \delta_a^b)  + 4\delta^2 \delta_a^b  - (v^c D_a (\delta)+ \delta \delta_a^c)( g_c\zeta^b + \delta \delta_c^b)) - 2 g_c\zeta^bU_a^c  + 2U_c^b v^c D_a (\delta)\]
\[ = 2((V - \delta v^c D_c)(v^b) - u^b )D_a (\delta)+2\zeta^b ( (V - \delta v^c D_c) (g_a)  - U_a^{\hspace{3pt}c} g_c) +(2(V - \delta v^cD_c)(\delta) + 3\delta^2) \delta_a^b  \]
\[ = (2V(\delta) + \delta^2) \delta_a^b.\]
So, passing to the shadow space, we get the formula:
\[ (T')^\beta_\alpha =  T^\beta_\alpha + \frac{1}{4}(2V(\delta) + \delta^2) \delta_\alpha^\beta.\]
We decompose  $T_\alpha^\beta$ and $(T')_\alpha^\beta$ as:
\[  (T')^\beta_\alpha =  (W')^\beta_\alpha + \frac{1}{4}X' \delta_\alpha^\beta, \hspace{10pt}  (W')^\alpha_\alpha =  0, \hspace{10pt} X' =    \frac{4}{n - 2}(T')^\alpha_\alpha, \]
\[  T^\beta_\alpha =  W^\beta_\alpha + \frac{1}{4} X \delta_\alpha^\beta, \hspace{10pt}  W^\alpha_\alpha =  0,  \hspace{10pt} X =    \frac{4}{n -2} T^\alpha_\alpha. \]
Then we have the key results:
\[ (W')^\beta_\alpha = W^\beta_\alpha, \]
\[ X' =  X +  2V(\delta) + \delta^2,\]
\[ (k -1) \delta = W(s)  = V(s) - q\delta, \hspace{10pt}\delta = (p - 1)^{-1}V(s), \]
\[ X' = X + 2(p - 1)^{-1}V^2(s) + (p -1)^{-2} (V(s))^2\]
\[ = X + 2(p - 1)^{-1}V(J^{-1} V(J)) + (p -1)^{-2} J^{-2}(V(J))^2, \]
\[ X' =  X + 2(p - 1)^{-1}J^{-1} V^2(J) - (p -1)^{-2} (2p - 3)J^{-2} (V(J))^2.\]

We have proved, in particular, that the trace-free part of the sectional curvature $W_\alpha^\beta$ on the shadow space is \emph{conformally invariant}.  By definition, this conformally invariant part is called the Weyl sectional curvature and is a straight-forward generalization of the sectional Weyl curvature of general relativity. Note that by judiciously choosing the conformal factor $J$, we may always assume (locally at least) that the trace of the sectional curvature is zero, so that $T_\alpha^\beta$ is trace-free.  The residual freedom in conformal rescaling then obeys the relation $2V(\delta) + \delta^2 = 0$, a first-order equation in $\delta$ but second-order in the conformal factor.  Finally note that in three dimensions, the sectional Weyl curvature vanishes identically, just as it does in three-dimensional metric diffferential geometry, since in three-dimensions, the shadow space $K$ is one-dimensional, so the sectional curvature (regarded as an endomorphism of $K$) is automatically proportional to the identity endomorphism, so is pure trace.  So non-trivial Weyl curvature arises first in dimension four, just as in  metric differential geometry.

\subsection{Example: indefinite metrics}
We consider the case that $G = 2^{-1} g_{ab}(x)v^a v^b$, where $g_{ab}(x) = g_{ba}(x)$ and the metric $g_{ab}(x)$ is invertible and indefinite;  these comprise the Kleinian (ultra-hyperbolic) metrics and the Lorentzian (hyperbolic) case of general relativity.  Then:
\[ g_a = g_{ab}v^b, \hspace{10pt} g_{ab} = g_{ab}, \hspace{10pt} g_{abc}  = 0, \hspace{10pt} G_a = 2^{-1} v^b v^c \partial_a g_{bc}, \]
\[ u^a g_{ab} = G_b - v^c \partial_c g_b = 2^{-1} v^a v^c \partial_b g_{ac} - v^av^c \partial_c g_{ab}\]
\[ =  2^{-1} v^c v^d(\partial_b g_{cd} -  2\partial_{(c} g_{d)b}) = - v^c v^d\Gamma_{cd}^{\hspace{8pt}a} g_{ab},\]
\[  \Gamma_{bc}^{\hspace{8pt}a} =  - 2^{-1}g^{ad}(\partial_d g_{bc} -  2\partial_{(b} g_{c)d}).\]
So $\Gamma_{cd}^{\hspace{8pt}a} $ are the usual Christoffel symbols of the Levi-Civita connection. Then:
\[ u^a = - v^b v^c\Gamma_{bc}^{\hspace{8pt}a},\hspace{10pt}   U_{b}^{\hspace{3pt}a} = - \frac{1}{2} D_b u^a =  v^c \Gamma_{bc}^{\hspace{10pt}a}, \]
\[ V = v^a \partial_a  - v^b v^c\Gamma_{bc}^{\hspace{8pt}a}D_a, \hspace{10pt} H_a = \partial_a - v^c \Gamma_{ab}^{\hspace{8pt}c}D_c.\]
So $V$ is the ordinary null geodesic spray, whereas the vector fields $H_a$ represent the horizontal vector fields on the tangent bundle of the standard Levi-Civita connection.   A conventional conformal transformation has $J(x, v) = J(x)\ne 0$, so $j_a = 0$ and $j_{ab} = 0$.  Also $p = k = 2$ and $q = 0$ and we have:
\[ H = \frac{1}{2}h_{ab} v^a v^b = G J(x), \hspace{10pt} h_{ab} = J(x) g_{ab}, \]
\[ u^a = \gamma^a + G\zeta^a + G^2 \xi^a + G^3 \sigma^a +  (\delta + G\eta + G^2 \rho) v^a, \]
\[ \gamma^a = h^{ab}(H_b - v^c \partial_c h_b), \hspace{10pt}W = v^a \partial_a + \gamma^a D_a, \hspace{10pt} \delta =   J^{-1}v^a J_a, \]
\[ \zeta^a =  - h^{ab}J_b,\hspace{10pt} \xi^a = \sigma^a =0, \hspace{10pt} \eta  = \rho = 0, \]
\[  u^a =  h^{ab}(H_b - v^c \partial_c h_b) - HJ^{-1}h^{ab}J_b+  v^a  J^{-1}v^b J_b.\]
Write $\Delta_{bc}^{\hspace{8pt} a}$ for the Christoffel symbols of $h_{ab}$.  Then we have:
\[ - \Gamma_{bc}^{\hspace{8pt} a} v^bv^c = v^b v^c( -  \Delta_{bc}^{\hspace{8pt} a}   - 2^{-1} h_{bc} J^{-1} h^{ad}J_d +  \delta^a_{(b}  J^{-1} J_{c)}), \]
\[ \Gamma_{bc}^{\hspace{8pt} a}  =  \Delta_{bc}^{\hspace{8pt} a}   + \frac{1}{2}J^{-1} h_{bc} h^{ad}J_d - J^{-1} J_{(b}\delta^a_{c)}. \]
This agrees with the standard transformation law for  Christoffel symbols:
\[  \Delta_{bc}^{\hspace{8pt} a} = - 2^{-1}h^{ad}(\partial_d h_{bc} -  2\partial_{(b} h_{c)d})  =  - 2^{-1}J^{-1} g^{ad}(\partial_d (Jg_{bc}) -  2\partial_{(b}( Jg_{c)d})) \]
\[ = \Gamma_{bc}^{\hspace{8pt} a} - \frac{1}{2}J^{-1} g^{ad}(g_{bc}J_d -  2J_{(b} g_{c)d}) = \Gamma_{bc}^{\hspace{8pt} a} - \frac{1}{2}J^{-1} h^{ad}h_{bc}J_d + J^{-1}J_{(b} \delta_{c)}^a. \]

\subsection{The Lie derivatives of $\lambda$}
We have in local co-ordinates: 
\[ \lambda =   dx^a\otimes D_a.\]
Then we have:
\[  \mathscr{L}_V \lambda =  dv^a\otimes D_a + dx^a \otimes (- \partial_a -D_a u^b) D_b, \]
\[ I = dv^a\otimes D_a + dx^a \otimes  \partial_a, \]
\[ \Gamma = \frac{1}{2}(I + \mathscr{L}_V \lambda) = (dv^a - \frac{1}{2}dx^b D_b u^a)\otimes D_a = \phi^a \otimes D_a, \]
\[ \phi^a = dv^a + U_b^a dx^b, \]
\[ \mathscr{L}_V \Gamma = \frac{1}{2}\mathscr{L}_V^2 \lambda =  \mathscr{L}_V (\phi^a \otimes D_a)\]
\[ = \iota_V (d\phi^a) \otimes D_a - \phi^a \otimes (\partial_a - 2U_a^b D_b)\]
\[ = V (U_b^a) dx^b \otimes D_a - v^b(dU_b^a) \otimes D_a - \phi^a \otimes (\partial_a - 2U_a^b D_b)\]
\[ = V (U_b^a) dx^b \otimes D_a - d (v^bU_b^a) \otimes D_a +  U_b^a (dv^b) \otimes D_a   - \phi^a \otimes (H_a - U_a^b D_b)\]
\[ = V (U_b^a) dx^b \otimes D_a + (du^a)\otimes D_a +  U_b^a (\phi^b - U^b_c dx^c) \otimes D_a   - \phi^a \otimes (H_a - U_a^b D_b)\]
\[ = V (U_b^a) dx^b \otimes D_a -2 \phi^b(U_b^a)\otimes D_a  + dx^b(H_b u^a)\otimes D_a+  U_b^a (\phi^b - U^b_c dx^c) \otimes D_a   - \phi^a \otimes (H_a - U_a^b D_b)\]
\[ = (V (U_b^a)   + H_b u^a -  U_c^a  U^c_b) dx^b \otimes D_a   - \phi^a \otimes H_a\]
\[ = T_b^a dx^b \otimes D_a   - \phi^a \otimes H_a.\]
So we have:
\[  \mathscr{L}_V \Gamma =  \frac{1}{2}\mathscr{L}_V^2 \lambda =  T_b^a dx^b \otimes D_a   - \phi^a \otimes H_a.\]

\bibliography{contactdynamics}{}
\bibliographystyle{plain}

\tableofcontents

\end{document}